\documentclass[a4paper,12pt]{article}
\usepackage[top=2cm, bottom=2cm, outer=2cm, inner=2cm, headsep=14pt]{geometry}
\usepackage{amsmath,amsfonts}
\usepackage{enumitem}
\usepackage{adjustbox}
\usepackage{blkarray}

\usepackage{multicol,color} 

\usepackage{tikz}
\usetikzlibrary{positioning, automata}
\usetikzlibrary{decorations.markings}

\usepackage{url} 

\usepackage{soul,xcolor} 

\tikzset{
  on each segment/.style={
    decorate,
    decoration={
      show path construction,
      moveto code={},
      lineto code={
        \path [#1]
        (\tikzinputsegmentfirst) -- (\tikzinputsegmentlast);
      },
      curveto code={
        \path [#1] (\tikzinputsegmentfirst)
        .. controls
        (\tikzinputsegmentsupporta) and (\tikzinputsegmentsupportb)
        ..
        (\tikzinputsegmentlast);
      },
      closepath code={
        \path [#1]
        (\tikzinputsegmentfirst) -- (\tikzinputsegmentlast);
      },
    },
  },
  mid arrow/.style={postaction={decorate,decoration={
        markings,
        mark=at position .5 with {\arrow[#1]{stealth}}
      }}},
}

\usepackage{avant}

\allowdisplaybreaks

\newcommand{\G}{\Gamma}
\def\nnu{{\boldsymbol\nu}}
\newcommand{\Mat}{\mbox{\rm Mat}}

\def\0{{\boldsymbol 0}}

\def\CC{{\mathbb C}}

\def\E{{\cal E}}

\def\U{{\cal U}}

\def\U{{\cal U}}
\def\RR{{\mathbb R}}

\def\NN{{\mathbb N}}
\def\ZZ{{\mathbb Z}}

\def\wh{\widehat}

\def\R{{\mathcal R}}

\def\B{{\mathcal B}}
\def\M{{\mathcal M}}

\def\ol{\overline}
\def\ds{\displaystyle}

\def\Ra{\Rightarrow}
\def\la{\leftarrow}
\def\La{\Leftarrow}
\def\ra{\rightarrow}
\def\XXi{{\mathfrak X}}

\DeclareMathOperator{\Span}{span}
\DeclareMathOperator{\dgr}{deg}
\DeclareMathOperator{\diag}{diag}
\DeclareMathOperator{\trace}{trace}

\DeclareMathOperator{\spec}{spec}

\DeclareMathOperator{\eig}{eig}

\DeclareMathOperator{\jj}{\boldsymbol{j}}

\DeclareMathOperator{\OO}{\boldsymbol{O}}

\newtheorem{theorem}{Theorem}[section]
\newtheorem{lemma}[theorem]{Lemma}
\newtheorem{corollary}[theorem]{Corollary}
\newtheorem{proposition}[theorem]{Proposition}
\newtheorem{definition}[theorem]{Definition}

\newtheorem{problem}{Problem}[section]

\definecolor{ForestGreen}{RGB}{12, 110, 46}
\definecolor{ForestGreenTwo}{RGB}{120, 110, 86}

\newenvironment{proof}{{\noindent\it Proof. }}{\nopagebreak\hspace*{0.5cm}\hfill$\hbox{\rule{3pt}{6pt}}$\smallskip}

\definecolor{ForestGreen}{RGB}{12, 110, 46}
\definecolor{ForestGreenTwo}{RGB}{120, 110, 86}

\newfont\fiverm{cmr5}
\def\eeq{\end{equation}} 
\def\lbeq#1{\begin{equation} \label{#1}} 
\def\gzit#1{{\rm (\ref{#1})}} 

\title{On Hoffman polynomials of $\lambda$-doubly stochastic irreducible matrices and commutative association schemes}

\author{
{Giusy Monzillo}\\
{\small Faculty of Mathematics, Natural Sciences}\\
{\small and Information Technologies}\\
{\small University of Primorska}\\
{\small Muzejski trg 2, 6000 Koper, Slovenia }\\
{\small Giusy.Monzillo@famnit.upr.si} \and
{Safet Penji\'c}\\
{\small Faculty of Mathematics, Natural Sciences}\\
{\small and Information Technologies; and}\\
{\small Andrej Maru\v{s}i\v{c} Institute}\\
{\small University of Primorska}\\
{\small Muzejski trg 2, 6000 Koper, Slovenia }\\
{\small Safet.Penjic@iam.upr.si}
}

\begin{document}
\setstcolor{red}
\maketitle

\begin{abstract}
Let $\G$ denote a finite (strongly) connected regular (di)graph with adjacency matrix $A$. The {\em Hoffman polynomial} $h(t)$ of $\G=\G(A)$ is the unique polynomial of smallest degree satisfying $h(A)=J$, where $J$ denotes the all-ones matrix. Let $X$ denote a nonempty finite set. A nonnegative matrix $B\in\Mat_X(\RR)$ is called {\em $\lambda$-doubly stochastic} if $\sum_{z\in X} (B)_{yz}=\sum_{z\in X} (B)_{zy}=\lambda$ for each $y\in X$. In this paper we first show that there exists a polynomial $h(t)$ such that $h(B)=J$ if and only if $B$ is a $\lambda$-doubly stochastic irreducible matrix. 
This result allows us to define the Hoffman polynomial of a $\lambda$-doubly stochastic irreducible matrix.
 
Now, let $B\in\Mat_X(\RR)$ denote a normal irreducible nonnegative matrix, and $\B=\{p(B)\mid p\in\CC[t]\}$ denote the vector space over $\CC$ of all polynomials in $B$. 
Let us define a $01$-matrix $\wh{A}$ in the following way: $(\wh{A})_{xy}=1$ if and only if $(B)_{xy}>0$ $(x,y\in X)$. Let $\G=\G(\wh{A})$ denote a (di)graph with adjacency matrix $\wh{A}$, diameter $D$, and let $A_D$ denote the distance-$D$ matrix of $\G$. We show that $\B$ is the Bose--Mesner algebra of a commutative $D$-class association scheme if and only if $B$ is a normal $\lambda$-doubly stochastic matrix with $D+1$ distinct eigenvalues and $A_D$ is a polynomial in $B$.
\end{abstract}


\smallskip
{\small



}


\section{Introduction}
\label{1A}


Let $B\in\Mat_X(\RR)$ denote a normal $\lambda$-doubly stochastic irreducible matrix, $\G$ denote the underlying weighted digraph of $B$, and $\B=\{p(B)\mid p\in\CC[t]\}$ denote the vector space over $\CC$ of all polynomials in $B$. In this paper, we study connections between commutative association schemes and $\lambda$-doubly stochastic matrices by considering the following question: under which combinatorial or algebraic restriction on $\G$ is the vector space $\B$ the Bose--Mesner algebra of a commutative association scheme? Formal definitions are given in Section~\ref{2B}.

We first give the relevant background before presenting our main results. Let $X$ denote a finite set, and $\Mat_X(\CC)$ the set of complex matrices with rows and columns indexed by $X$. Let $\R=\{R_0,R_1,\ldots,R_d\}$ denote a set of cardinality $d+1$ of nonempty subsets of $X\times X$. The elements of the set $\R$ are called {\em relations} (or {\em classes}) on $X$. For each integer $i$ $(0\le i\le d)$, let $B_i\in\Mat_X(\CC)$ denote the adjacency matrix of the graph $(X,R_i)$ (directed, in general). The pair ${\XXi}=(X,\R)$ is a {\em commutative $d$-class association scheme} if the {\em relation matrices} $B_i$ satisfy the following properties
\begin{enumerate}[leftmargin=1.7cm,label=\rm(AS\arabic*)]
\item $B_0=I$, the identity matrix.\label{ce}
\item $\ds{\sum_{i=0}^d B_i=J}$, the all-ones matrix.\label{cg}
\item ${B_i}^\top\in\{B_0,B_1,\ldots,B_d\}$ for $0\le i\le d$.
\item $B_iB_j$ is a linear combination of $B_0,B_1,\ldots,B_d$ for $0\le i,j\le d$ (i.e., for every $i,j$ $(0\le i,j\le d)$ there exist positive integers $p^h_{ij}$ $(0\le h\le d)$, known as {\em intersection numbers}, such that $B_iB_j=\sum_{h=0}^d p^h_{ij} B_h$).\label{cb}
\item $B_iB_j=B_jB_i$ for every $i,j$ $(0\le i,j\le d)$ (i.e., for the intersection numbers $p^h_{ij}$, $0\le i,j,h\le d$, from \ref{cb} we have that $p^h_{ij}=p^h_{ji}$).\label{cf}
\end{enumerate}

By \ref{ce}--\ref{cf} the vector space $\M=\Span\{B_0,B_1,\ldots,B_d\}$ is a commutative algebra; it is known as the {\em Bose--Mesner algebra} of $\XXi$. We say that a matrix $B$ {\em generates} $\M$ if every element in $\M$ can be written as a polynomial in $B$. We say that $\XXi$ is {\em symmetric} if the $B_i$'s are symmetric matrices.

A nonnegative matrix $B\in\Mat_X(\RR)$ such that $\sum_{z\in X} (B)_{yz}=\sum_{z\in X} (B)_{zy}=\lambda$ for each $y\in X$ is called a {\em $\lambda$-doubly stochastic matrix}. If $\lambda=1$, the matrix is simply called {\em doubly stochastic}. The following result was proved by {\sc Birkhoff} (1946) and independently by {\sc von Neumann} (1953): Each doubly stochastic matrix $B$ can be represented as a convex combination of permutation matrices, that is, 
\begin{equation}
\label{Qd}
B=c_1 P_1 + c_2 P_2 + \cdots + c_m P_m
\end{equation}
where the $c_i$'s are positive real numbers with $\sum_{i=1}^m c_i = 1$, and $P_1,P_2,\ldots,P_m$ are distinct permutation matrices. It is well known that the convex representation \eqref{Qd} of the doubly stochastic matrix $B$ is unique (up to reordering the terms) if and only if the graph $\G$ is uniquely edge colourable, where $\G$ is a bipartite graph with bipartition $(V_1,V_2)$, where $V_1=\{x_1,x_2,\ldots,x_{|X|}\}$, $V_2=\{y_1,y_2,\ldots,y_{|X|}\}$ and two  vertices $x_i$ and $y_j$ are joined by $\left(\sum_{i=1}^m P_i\right)_{ij}$ edges $(1 \le i,j \le |X|)$  (see, for example, \cite[Subchapter~9.2]{AD}). In \cite{DU}, {\sc Dufoss{\'{e}}} and {\sc U{\c{c}}ar} showed that determining the minimal number of permutation matrices needed in \eqref{Qd} is strongly NP-complete. Some interesting papers that study doubly stochastic matrices are, for example, \cite{BG,BD,BGb,PM,SK}. With respect to representation \eqref{Qd}, from our point of view, it would be interesting to study the combinatorial structure of a (di)graph with adjacency matrix $\sum_{i=1}^m P_i$. In this paper we study representation~\eqref{Qd} in the set-up of Problem~\ref{Aa}.

\begin{problem}{\label{Aa}}{\rm
Let $B\in\Mat_X(\RR)$ denote a non-negative matrix. Assume that the matrix $B$ has exactly $m+1$ distinct entries $\{0,c_1,\ldots,c_m\}$, so that we can write $B$ as a linear combination of $01$-matrices $F_i$ $(1\le i\le m)$ as follows
$$
B=c_1 F_1 + c_2 F_2 + \cdots + c_m F_m
$$
(note that the $c_i$'s are positive real numbers and our $F_i$'s are not necessarily permutation matrices). We also assume that $F_i$'s are $\circ$-idempotents, i.e., $F_i\circ F_j=\OO$ whenever $i\ne j$, where $\circ$ denotes the elementwise-Hadamard product. Let $A=\sum_{i=0}^m F_i$. Can we describe the combinatorial structure (or give some algebraic properties) of the digraph $\G=\G(A)$, so that the vector space $\B=\{p(B)\mid p\in\CC[t]\}$ over $\CC$ of all polynomials in $B$ is the Bose--Mesner algebra of a commutative association scheme? Also, what can we say about the entries of the matrix $B$?
}\end{problem}

Since the all-$1$ matrix $J$ belongs to every commutative association scheme, as a first sub-problem of Problem~\ref{Aa}, we are interested in the case when $J$ is a polynomial in $B$. This property implies that $B$ is a $\lambda$-doubly stochastic matrix (see Theorem~\ref{Qe}), so we obtain an answer to the  second part of the problem. Let us give some background in this direction. For the moment, let $\G$ denote an undirected graph with vertex set $X$, adjacency matrix $A$, and let $J$ denote the all-ones matrix of order $|X|$. In \cite{HoP}, {\sc Hoffman} proved that there exists a polynomial $p(x)$ such that 
\begin{equation}
\label{Qa}
p(A) = J
\end{equation}
if and only if $\G$ is connected and regular. In \cite{HMc}, {\sc Hoffman} and {\sc McAndrew} studied the case of a directed graph, and obtained a similar result: there exists a polynomial $p(x)$ such that \eqref{Qa} holds if and only if $\G$ is strongly connected and regular. Moreover, they showed that the unique polynomial of smallest degree satisfying \eqref{Qa} is $h(t)=\frac{|X|}{q(k)}q(t)$, where $\G=\G(A)$ is a regular digraph of valency $k$, and $(t- k)q(t)$ is the minimal polynomial of $A$. Next, it is well known that a digraph $\G$ is strongly connected if and only if its adjacency matrix $A$ is irreducible (see, for example, \cite[Section~8.3]{MCm}). For the moment, let $C\in\Mat_X(\RR)$ denote a nonnegative matrix. In \cite{WD}, {\sc Wu} and {\sc Deng} study a polynomial that sends a nonnegative irreducible matrix to a positive rank one matrix;  they showed that there is a polynomial $p(t)\in\RR[t]$ such that $p(C)$ is a positive matrix of rank one if and only if $C$ is irreducible. Moreover, they show that the lowest degree of such a polynomial $p(t)$ with $\trace p(C) = |X|$ is unique. The first main result of our paper is Theorem~\ref{Qe}, which is in the same spirit as that of {\sc Hoffman} and {\sc McAndrew} from \cite{HMc} (note that one direction of our theorem also follows from \cite[Theorem~2.2]{WD}).

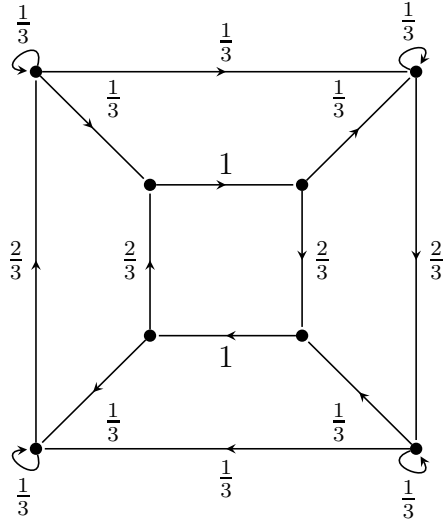
\begin{figure}[t!]
\begin{center}
\begin{tikzpicture}[-,>=stealth,shorten >=1pt,auto,node distance=2.5cm,semithick]
\tikzstyle{every state}=[draw=white, circle, minimum size=4pt, inner sep=0pt]
\node[state] () at (-1,-1) {$B=\left(\begin{array}{cccccccc} \frac{1}{3} & 0 & 0 & \frac{2}{3} & 0 & 0 & 0 & 0\\ \frac{1}{3} & \frac{1}{3} & 0 & 0 & 0 & \frac{1}{3} & 0 & 0\\ 0 & \frac{2}{3} & \frac{1}{3} & 0 & 0 & 0 & 0 & 0\\ 0 & 0 & \frac{1}{3} & \frac{1}{3} & 0 & 0 & 0 & \frac{1}{3}\\ \frac{1}{3} & 0 & 0 & 0 & 0 & 0 & 0 & \frac{2}{3}\\ 0 & 0 & 0 & 0 & 1 & 0 & 0 & 0\\ 0 & 0 & \frac{1}{3} & 0 & 0 & \frac{2}{3} & 0 & 0\\ 0 & 0 & 0 & 0 & 0 & 0 & 1 & 0 \end{array}\right)
$
};
\end{tikzpicture}~~~~~~~~
\begin{tikzpicture}[-,>=stealth,shorten >=1pt,auto,node distance=2.5cm,semithick]
    \tikzstyle{every state}=[draw=black, circle, fill=black, minimum size=4pt, inner sep=0pt]
    \node[state] (A) at (-1,-1) {};
    \node[state] (B) at (1,-1) {};
    \node[state] (C) at (1,1) {};
    \node[state] (D) at (-1,1) {};
    \node[state] (E) at (-2.5,-2.5) {};
    \node[state] (F) at (2.5,-2.5) {};
    \node[state] (G) at (2.5,2.5) {};
    \node[state] (H) at (-2.5,2.5) {};
	
\path (H) edge [out=75, in=175, loop] node [pos=0.5, above] {$\frac{1}{3}$} (H);
\path (G) edge [out=-200, in=-300, loop] node [pos=0.5, above] {$\frac{1}{3}$} (G);
\path (E) edge [out=-75, in=-175, loop] node [pos=0.5, below] {$\frac{1}{3}$} (E);
\path (F) edge [out=200, in=300, loop] node [pos=0.5, below] {$\frac{1}{3}$} (F);

    \path (B) edge [postaction={decorate,decoration={markings,mark=at position 0.5 with {\arrow{>}}}}] node {$1$} (A);
    \path (F) edge [postaction={decorate,decoration={markings,mark=at position 0.5 with {\arrow{>}}}}] node {$\frac{1}{3}$} (B)
	edge [postaction={decorate,decoration={markings,mark=at position 0.5 with {\arrow{>}}}}] node {$\frac{1}{3}$} (E);
    \path (C) edge [postaction={decorate,decoration={markings,mark=at position 0.5 with {\arrow{>}}}}] node {$\frac{1}{3}$} (G);
    \path (G) edge [postaction={decorate,decoration={markings,mark=at position 0.5 with {\arrow{>}}}}] node {$\frac{2}{3}$} (F);
    \path (H) edge [postaction={decorate,decoration={markings,mark=at position 0.5 with {\arrow{>}}}}] node {$\frac{1}{3}$} (D)
	edge [postaction={decorate,decoration={markings,mark=at position 0.5 with {\arrow{>}}}}] node {$\frac{1}{3}$} (G);
	\path (E) edge [postaction={decorate,decoration={markings,mark=at position 0.5 with {\arrow{>}}}}] node {$\frac{2}{3}$} (H);
	\path (A) edge [postaction={decorate,decoration={markings,mark=at position 0.5 with {\arrow{>}}}}] node {$\frac{2}{3}$} (D)
	          edge [postaction={decorate,decoration={markings,mark=at position 0.5 with {\arrow{>}}}}] node {$\frac{1}{3}$} (E);
	\path (D) edge [postaction={decorate,decoration={markings,mark=at position 0.5 with {\arrow{>}}}}] node {$1$} (C);
	\path (C) edge [postaction={decorate,decoration={markings,mark=at position 0.5 with {\arrow{>}}}}] node {$\frac{2}{3}$} (B);
\end{tikzpicture}
\caption{A doubly stochastic matrix $B$ and its underlying weighted digraph. The Hoffman polynomial of $B$ is $h(t)=\frac{8}{q(1)}q(t)$, where $q(t)=t^7 - \frac{1}{3}t^6 + \frac{1}{3}t^5 + \frac{5}{27}t^4 - \frac{8}{27}t^3 + \frac{8}{27}t^2 - \frac{32}{243}t$.}
\label{9r}
\end{center}
\end{figure}

\begin{theorem}
\label{Qe}
For a nonnegative matrix $B\in\Mat_X(\RR)$ there exists a polynomial $p\in\CC[t]$ such that 
\begin{equation}
\label{Qf}
p(B) = J
\end{equation}
if and only if $B$ is a $\lambda$-doubly stochastic irreducible matrix. Moreover, the unique polynomial of smallest degree satisfying \eqref{Qf} is $h(t)=\frac{|X|}{q(\lambda)}q(t)$,  where $q(\lambda)\ne0$ and $(t-\lambda)q(t)$ is the minimal polynomial of $B$. 
\end{theorem}

We call the polynomial $h(t)$ from the Theorem~\ref{Qe} {\em Hoffman polynomial of} $B$. We use Theorem~\ref{Qe} to prove Theorem~\ref{Qg}, 
giving an algebraic-combinatorial characterization when the Bose--Mesner algebra of a commutative association scheme is generated by a normal $\lambda$-doubly stochastic matrix (for results about when the Bose--Mesner algebra of a commutative association scheme is generated by a (directed) graph, see \cite{FQpG,MPc,TtYl}). For a normal $\lambda$-doubly stochastic irreducible matrix $B$ with $d+1$ distinct eigenvalues $\{\lambda,\lambda_1,\ldots,\lambda_d\}$ the Hoffman polynomial is $h(t)=\frac{|X|}{\pi_0}\prod_{i=1}^d (t-\lambda_i)$, where $\pi_0=\prod_{i=1}^d (\lambda-\lambda_i)$. In Section~\ref{oa}, using the inner product $\langle p,q\rangle=\frac{1}{|X|}\trace(p(B)\ol{q(B)}^\top)$ on the ring $\RR_d[t]$, we define the so-called ``predistance polynomials'' $\{p_i(t)\}_{i=0}^d$, and we show that $\sum_{i=0}^d p_i(A) = J$ (see Lemma~\ref{oi}). The term ``predistance polynomials'' is taken from the theory of distance-regular graphs (see, for example, \cite{FAb,FAc,FAe,FaD,FAf}). For the moment Problem~\ref{Aa} seems to be a hard problem, so we make one restriction on it:  we assume that the number of distinct eigenvalues of $B$ is $D+1$, where $D$ is the diameter of a graph $\G=\G(A)$. The motivation for this restriction (again) arises from the theory of distance-regular graphs (see, for example, \cite{CFGM,FaD,Fsp,SPm,SP}). As a consequence of our restriction, the second main result of this paper is the following theorem:

\begin{theorem}
\label{Qg}
Let $B\in\Mat_X(\RR)$ denote a nonnegative irreducible matrix. Let $\B=\{p(B)\mid p\in\CC[t]\}$ denote the vector space of all polynomials in $B$. Define a $01$-matrix $A$ in the following way: $(A)_{xy}=1$ if and only if $(B_i)_{xy}>0$. Let $\G=\G(A)$ denote a digraph with adjacency matrix $A$,  diameter $D$, and let $A_D$ denote the distance-$D$ matrix of $\G$. Then, $\B$ is the Bose--Mesner algebra of a commutative $D$-class association scheme if and only if $B$ is a normal $\lambda$-doubly stochastic matrix with $D+1$ distinct eigenvalues and $A_D$ is a polynomial in $B$.
\end{theorem}

Our Theorem~\ref{Qg} is an analogue of a result from algebraic graph theory, see for example {\cite[Proposition~2]{Fsp} or \cite{FGJ}}, where the authors considered an undirected graph (symmetric adjacency matrix) and proved the following claim: An undirected regular graph $\G=\G(A)$ with diameter $D$ and $d + 1$ distinct eigenvalues is a distance-regular if and only if $D = d$ and the distance-$D$ matrix $A_D$ is a polynomial in $A$.

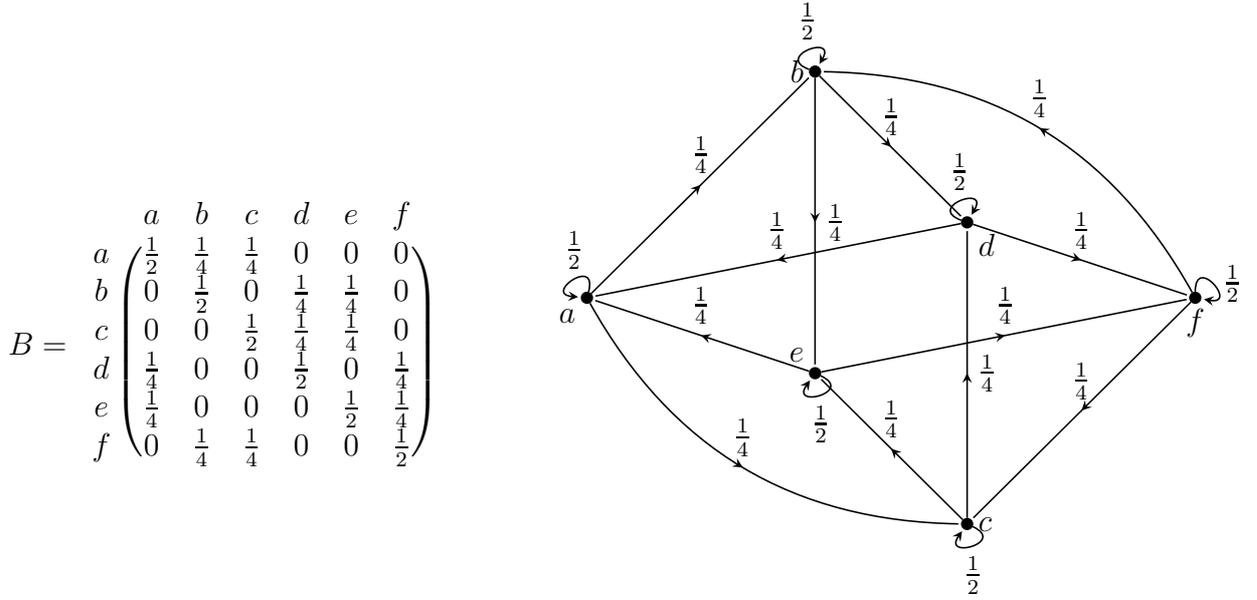
\begin{figure}[t!]
\begin{center}
\begin{tikzpicture}[-,>=stealth,shorten >=1pt,auto,node distance=2.5cm,semithick]
\tikzstyle{every state}=[draw=white, circle, minimum size=4pt, inner sep=0pt]
\node[state] () at (0,0) {$B=\begin{blockarray}{ccccccc}
~ & a & b & c & d & e & f \\
\begin{block}{c(cccccc)}
a & \frac{1}{2} & \frac{1}{4} & \frac{1}{4} & 0 & 0 & 0\\ 
b & 0 & \frac{1}{2} & 0 & \frac{1}{4} & \frac{1}{4} & 0\\ 
c & 0 & 0 & \frac{1}{2} & \frac{1}{4} & \frac{1}{4} & 0\\ 
d & \frac{1}{4} & 0 & 0 & \frac{1}{2} & 0 & \frac{1}{4}\\ 
e & \frac{1}{4} & 0 & 0 & 0 & \frac{1}{2} & \frac{1}{4}\\ 
f & 0 & \frac{1}{4} & \frac{1}{4} & 0 & 0 & \frac{1}{2}\\
\end{block}
\end{blockarray}$};
\end{tikzpicture}~~~~~~~~
\begin{tikzpicture}[-,>=stealth,shorten >=1pt,auto,node distance=2.5cm,semithick]
    \tikzstyle{every state}=[draw=black, circle, fill=black, minimum size=4pt, inner sep=0pt]
\node[state] (a) at (-4,0) {};
\node[state] (b) at (-1,3) {};
\node[state] (c) at (1,-3) {};
\node[state] (d) at (1,1) {};
\node[state] (e) at (-1,-1) {};
\node[state] (f) at (4,0) {};
\node at (-4,0) [anchor=north east] {$a$};
\node at (4,0) [anchor=north] {$f$};
\node at (-1,3) [anchor=east] {$b$};
\node at (-1,-1) [anchor=south east] {$e$};
\node at (1,1) [anchor=north west] {$d$};
\node at (1,-3) [anchor=west] {$c$};

\path (a) edge [out=75, in=175, loop] node [pos=0.5, above] {$\frac{1}{2}$} (a);
\path (b) edge [out=-200, in=-300, loop] node [pos=0.5, above] {$\frac{1}{2}$} (b);
\path (d) edge [out=-200, in=-300, loop] node [pos=0.5, above] {$\frac{1}{2}$} (d);
\path (c) edge [out=-25, in=-125, loop] node [pos=0.5, below] {$\frac{1}{2}$} (c);
\path (e) edge [out=-25, in=-125, loop] node [pos=0.5, below] {$\frac{1}{2}$} (e);
\path (f) edge [out=-270, in=-370, loop] node [pos=0.5, right] {$\frac{1}{2}$} (f);

\path (a) edge [postaction={decorate,decoration={markings,mark=at position 0.5 with {\arrow{>}}}}] node [pos=0.5, above] {$\frac{1}{4}$} (b);
\path (a) edge [bend right, postaction={decorate,decoration={markings,mark=at position 0.5 with {\arrow{>}}}}] node [pos=0.5, above] {$\frac{1}{4}$} (c);
\path (b) edge [postaction={decorate,decoration={markings,mark=at position 0.5 with {\arrow{>}}}}] node [pos=0.5, right] {$\frac{1}{4}$} (e);
\path (b) edge [postaction={decorate,decoration={markings,mark=at position 0.5 with {\arrow{>}}}}] node [pos=0.5, above] {$\frac{1}{4}$} (d);
\path (c) edge [postaction={decorate,decoration={markings,mark=at position 0.5 with {\arrow{>}}}}] node [pos=0.5, above] {$\frac{1}{4}$} (e);
\path (c) edge [postaction={decorate,decoration={markings,mark=at position 0.5 with {\arrow{>}}}}] node [pos=0.5, right] {$\frac{1}{4}$} (d);
\path (d) edge [postaction={decorate,decoration={markings,mark=at position 0.5 with {\arrow{>}}}}] node [pos=0.5, above] {$\frac{1}{4}$} (a);
\path (d) edge [postaction={decorate,decoration={markings,mark=at position 0.5 with {\arrow{>}}}}] node [pos=0.5, above] {$\frac{1}{4}$} (f);
\path (e) edge [postaction={decorate,decoration={markings,mark=at position 0.5 with {\arrow{>}}}}] node [pos=0.5, above] {$\frac{1}{4}$} (a);
\path (e) edge [postaction={decorate,decoration={markings,mark=at position 0.5 with {\arrow{>}}}}] node [pos=0.5, above] {$\frac{1}{4}$} (f);
\path (f) edge [bend right, postaction={decorate,decoration={markings,mark=at position 0.5 with {\arrow{>}}}}] node[pos=0.5, above] {$\frac{1}{4}$} (b);
\path (f) edge [postaction={decorate,decoration={markings,mark=at position 0.5 with {\arrow{>}}}}] node [pos=0.5, above] {$\frac{1}{4}$} (c);

\end{tikzpicture}
\caption{A normal doubly stochastic matrix $B$ and its underlying weighted digraph. The Hoffman polynimal of $B$ is $h(t)=16t^3 - 16t^2 + 8t - 2$, and the predistance polynomials are $p_0(t)=1$, $p_1(t)=4t-2$, $p_2(t)=8t^2 - 8t + 2$ and $p_3(t)=16t^3 - 24t^2 + 12t - 3$. By Lemma~\ref{oi}, $\sum_{i=0}^3 p_i(B)=J$. Moreover, $B$ generates a $3$-class association scheme.}
\label{9s}
\end{center}
\end{figure}

Our paper is organized as follows. In Section~\ref{2B}, we recall basic concepts from algebraic graph theory (experts from the field can skip this section). Our paper then starts from Section~\ref{ra}, where we prove Theorem~\ref{Qe}. In Section~\ref{OA} we define predistance-polynomials, the polynomials that we use later in the paper. In Section~\ref{oa}, we prove Theorem~\ref{Qg}.


\section{Preliminaries}
\label{2B}

A {\em digraph} with {\em vertex set} $X$ and {\em arc set} $\E$ is a pair $\G=(X, \E)$ which consists of a finite set $X=X(\G)$ of {\em vertices} and a set $\E = \E(\G)$ of {\em arcs} ({\em directed edges}) between vertices of $\G$. As the initial and final vertices of an arc are not necessarily different, digraphs may have loops (arcs from a vertex to itself) and multiple arcs, that is, there can be more than one arc from each vertex to any other. If $e = (x,y)\in\E$ is an arc from $x$ to $y$, then the vertex $x$ (and the arc $e$) is {\em adjacent to} the vertex $y$, and the vertex $y$ (and the arc $e$) is {\em adjacent from} $x$. The {\em converse directed graph} $\ol{\G}$ is obtained from $\G$ by reversing the direction of each arc. 
For a vertex $x$, let $\G_1^{\la}(x)$ (resp. $\G_1^{\ra}(x)$) denote the set of vertices adjacent to (resp. from) the vertex $x$. In other words,
$$
\G_1^{\ra}(x) = \{z\mid (x,z)\in \E(\G)\}
\qquad\mbox{ and }\qquad
\G_1^{\la}(x) = \{z\mid (z,x)\in \E(\G)\}.
$$
Two small comments about the above notation: (i) drawing a directed edge from $x$ to $z$, we have $x\ra z$, which yields the idea of using the notation $\G_1^{\ra}(x)$; (ii) drawing a  directed edge from $z$ to $x$, we have $x\la z$ (or $z\ra x$), which yields the idea of using the notation $\G_1^{\la}(x)$. The elements of $\G_1^{\ra}(x)$ are called {\em neighbors} of $x$. Instead of a set of vertices, we can consider a set of arcs: for a vertex $y$, let $\delta_1^{\la}(y)$ (resp. $\delta_1^{\ra}(y)$) denote the set of arcs adjacent to (resp. from) the vertex $y$. The number $|\delta_1^{\ra}(y)|$ is called the {\em out-degree of $y$} and is equal to the number of edges leaving $y$. The number $|\delta_1^{\la}(y)|$ is called the {\em in-degree of $y$} and is equal to the number of edges going to $y$. A digraph $\G$ is {\em $k$-regular} (of valency $k$) if $|\delta^{\ra}_1(y)| = |\delta_1^{\la}(y)| = k$ for all $y\in X$. We call $\G$ {\em simple} if $\G$ contains neither loops nor multiple edges. 

Let $\G = (X,\E)$ denote a digraph. For any two vertices $x, y \in X$, a {\em directed walk} of length $h$ from $x$ to $y$ is a sequence $[x_0,x_1,x_2,\ldots,x_h]$ $(x_i\in X,\, 0\le i\le h)$ such that $x_0 = x$, $x_h = y$, and $x_i$ is adjacent to $x_{i+1}$ (i.e. $x_{i+1}\in\G^{\ra}_1(x_i)$) for $0\le i\le h-1$. We say that $\G$ is {\em strongly connected} if for any $x, y\in X$ there is a directed walk from $x$ to $y$. A {\em closed directed walk} is a directed walk from a vertex to itself. A {\em directed path} is a directed walk such that all vertices of the directed walk are distinct. A {\em cycle} is a closed directed path. The {\em girth} of $\G$ is the length of a shortest cycle in $\G$.

For any $x, y\in X$, the {\em distance} from $x$ to $y$ (or between $x$ and $y$), denoted by $\partial(x, y)$, is the length of a shortest directed path from $x$ to $y$. The {\em diameter} $D = D(\G)$ of a strongly connected digraph $\G$ is defined to be 
$$
D = \max\{\partial(y,z)\,|\,y, z\in X\}.
$$
For a vertex $x\in X$ and any nonnegative integer $i$ not exceeding $D$, let $\G^\ra_i(x)$ (or $\G_i(x)$) denote the subset of vertices in $X$ that are at distance $i$ from $x$, i.e.,
$$
\G^\ra_i(x)=\{z\in X\mid \partial(x,z)=i\}.
$$
We also define the set $\G^\la_i(x)$ as $\G^\la_i(x)=\{z\in X\mid \partial(z,x)=i\}$. Let $\G_{-1}(x) = \G_{D+1}(x) := \emptyset$. The {\it eccentricity} of $x$, denoted by $\varepsilon=\varepsilon(x)$, is the maximum distance between $x$ and any other vertex of $\G$. Note that the diameter of $\G$ equals $\max\{\varepsilon(x)\mid x\in X\}$.

All undirected graphs in this paper can be understood as digraphs in which an undirected edge between two vertices $x$ and $y$ represents two arcs, an arc from $x$ to $y$, and an arc from $y$ to $x$. In diagrams, instead of drawing two arcs, we draw one undirected edge between vertices $x$ and $y$. For a basic introduction to the theory of undirected graphs we refer to \cite[Section~2]{FMPS}. With the word {\em graph} we refer to a finite simple digraph.

\subsection{Doubly stochastic matrix}

Let $X$ denote a nonempty finite set and let $\RR$ (resp. $\CC$) denote the real number field (resp. the complex number field). Let $\Mat_X(\RR)$ (resp. $\Mat_X(\CC)$) denote the $\RR$-algebra (resp. the $\CC$-algebra) consisting of all matrices whose rows and columns are indexed by $X$ and whose entries are in $\RR$ (resp. $\CC$).

A square matrix $B\in\Mat_{X}(\RR)$ is said to be a {\em $\lambda$-doubly stochastic} if $B\ge \OO$ and $B\jj=B^\top\jj=\lambda\jj$, where $\OO$ is the zero square matrix of order $|X|$, $B\ge\OO$ is a shortcut for $(B)_{xy}\ge 0$ (for all $x,y\in X$), $\jj$ is the $|X|$-dimensional column-vector with $1$ in all entries, and $B^\top$ is the transpose of $B$. If $\lambda=1$, the matrix is called {\em doubly stochastic}. A {\em permutation matrix} $P$ is a square matrix with exactly one $1$ in each row and column, and the rest of the entries being zero.

\subsection{Elementary algebraic graph theory}

In this section, we recall some definitions and basic concepts from algebraic graph theory.

The {\em adjacency matrix} $A\in\Mat_X(\CC)$ of a digraph $\G$ (with vertex set $X$) is indexed by the vertices from $X$, and defined in the following way
\begin{equation}
\label{gs}
\mbox{$(A)_{yz} =$ the number of arcs from $y$ to $z$}\qquad (y,z\in X)
\end{equation}
(note that $(A)_{yz}\in\ZZ^+_0$). The distance-$i$ matrix $A_i$ $(2\le i\le D)$ of a digraph $\G$ with diameter $D$ and vertex set $X$ is defined by
$$
(A_i)_{zy}=\left\{\begin{matrix}
1&\mbox{ if $\partial(z,y)=i$},\\
0&\mbox{ otherwise.~~~~}
\end{matrix}\right. \qquad(z,y\in X,~2\le i\le D).
$$
We also define $A_0=I$ and $A_1=A$. A matrix $B\in\Mat_X(\CC)$ is said to be {\em reducible} when there exists a permutation matrix $P$ such that $P^\top B P=\left(
\begin{matrix}
X&Y\\ \OO&Z
\end{matrix}
\right)$, where $X$ and $Z$ are both square, and $\OO$ is a zero matrix of suitable size. Otherwise, $B$ is said to be {\em irreducible}.

\begin{theorem}[Perron--Frobenius Theorem]
\label{gC}
Let $B$ denote an irreducible nonnegative matrix, and let $\eig(B)$ denote the set of
distinct eigenvalues of $B$. If $\theta=\max\limits_{\lambda\in\eig(B)}|\lambda|$, then the following hold.
\begin{enumerate}[label=\rm(\roman*)]
\item $\theta\in\eig(B)$ and $\theta>0$.
\item The algebraic multiplicity of $\theta$ is equal to $1$.
\item There exists an eigenvector $\nnu$ with all positive entries, such that $B\nnu = \theta \nnu$.
\end{enumerate}
Sometimes it is useful to normalize a vector $\nnu$ from {\rm (iii)} in such a way that the smallest entry is equal to $1$. Such a vector $\nnu$ is called a Perron--Frobenius eigenvector.
\end{theorem}

\begin{proof}
See, for example, \cite[Section~8.3]{MCm}.
\end{proof}

\begin{lemma}[{see, for example, \cite[Section~8.3]{MCm}}]
\label{gb}
A digraph $\G$ with adjacency matrix $A$ is strongly connected if and only if $A$ is an irreducible matrix.
\end{lemma}


\begin{corollary}
\label{gc}
Let $\G=\G(A)$ denote a simple strongly connected digraph, and let $\eig(\G)$ denote the set of distinct eigenvalues of $\G$. If $\theta=\max_{\lambda\in\spec(\G)}|\lambda|$, then the following hold.
\begin{enumerate}[label=\rm(\roman*)]
\item $\theta\in\eig(\G)$ and $\theta>0$.
\item The algebraic multiplicity of $\theta$ is equal to $1$.
\item There exists an eigenvector $\nnu$ with all positive entries, such that $A\nnu = \theta \nnu$.
\end{enumerate}
\end{corollary}

\begin{proof}
Routine using Lemma~\ref{gb}. (See, for example, \cite[Section~8.3]{MCm}.)
\end{proof}

\smallskip
A matrix $A\in\Mat_X(\CC)$ is called {\em normal} if it commutes with its adjoint, i.e. if $A\ol{A}^\top=\ol{A}^\top A$.

\begin{theorem}[{see, for example, \cite[Chapter~7]{AS}}]
\label{gd}
Let $A\in\Mat_X(\CC)$ denote a matrix over $\CC$, with rows and columns indexed by $X$. Then, the following are equivalent.
\begin{enumerate}[label=\rm(\roman*)]
\item $A$ is normal.
\item $\CC^{|X|}$ has an orthonormal basis consisting of eigenvectors of $A$.
\item $A$ is a diagonalizable matrix.
\item The algebraic multiplicity of $\lambda$ is equal to the geometric multiplicity of $\lambda$, for every eigenvalue $\lambda$ of $A$.
\end{enumerate}
\end{theorem}

Two matrices $A,B\in\Mat_X(\CC)$ are said to be {\em simultaneously diagonalizable} if there is a nonsingular $S\in\Mat_X(\CC)$ such that $S^{-1}AS$ and $S^{-1}BS$ are both diagonal.

\begin{lemma}[{\cite[Theorem 1.3.12]{HJ}}]
\label{ge}
Two diagonalizable matrices are simultaneously diagonalizable if and only if they commute.
\end{lemma}

\begin{theorem}
\label{gf}
Let $\M$ denote a space of commutative normal matrices. Then, there exists a unitary matrix $U\in\Mat_X(\CC)$ which diagonalizes $\M$.
\end{theorem}

\begin{proof}
Immediate from Theorem~\ref{gd}, Lemma~\ref{ge} and \cite[Subsection~1.3]{HJ}.
\end{proof}


\subsection{Underlying digraph of a nonnegative matrix $B$}

The {\em underlying digraph of a nonnegative matrix $B\in\Mat_X(\RR)$} is defined as a pair $\G=(X,E)$, in which $X$ denotes the set of vertices (nodes), and $E$ stands for the set of arcs such that $(x,y)\in E$ if and only if $(B)_{xy}> 0$. With other words, the adjacency matrix $A$ of an underlying digraph of a nonnegative matrix $B$ is defined in the following way:
$$
(A)_{xy}=
\left\{
\begin{array}{ll}
1&\mbox{ if $B_{xy}>0$},\\
0&\mbox{ otherwise.}
\end{array}
\right.\qquad(x,y\in X).
$$

\begin{lemma}
\label{gr}
Let $B\in\Mat_X(\RR)$ denote a nonnegative matrix, and let $A$ denote the adjacency matrix of the underlying digraph of $B$. Then $B$ is irreducible if and only if $A$ is irreducible.
\end{lemma}

\begin{proof}
Routine.
\end{proof}

\subsection{Underlying weighted digraph of a nonnegative matrix $B$}

A {\em weighted digraph} is a digraph whose arcs are assigned values, known as {\em weights}. An {\em underlying weighted digraph of a nonnegative matrix $B$} is defined as a triplet $\Delta=(X,E,\omega)$ for which the following {\rm(i)--(iii)} holds.
\begin{enumerate}[label=\rm(\roman*)]
\item $X$ denotes the set of vertices.
\item $E$ stands for the set of arcs such that $(x,y)\in E$ if and only if $B_{x,y}>0$.
\item $\omega:E \ra \RR^+_0$ stands for a function that weights each arc of the graph, which is defined in the following way: $\omega(x,y)=(B)_{xy}$ $((x,y)\in E)$.
\end{enumerate} 
Note that $\Delta$ is the underlying digraph of a nonnegative matrix $B$ such that each arc $(x,y)$ of $\Delta$ has weight $(B)_{xy}$.

\subsection{Number of walks in $\G$ and $B^\ell$}

Lemma~\ref{hB} is well known for undirected graphs. We give a corresponding claim for directed graphs, in particular, for our definition of the adjacency matrix $A$ of a digraph $\G$ as given in \eqref{gs}.

\begin{lemma}
\label{hB}
Let $\G=\G(A)$ denote a strongly connected digraph with vertex set $X$, diameter $D$, and adjacency matrix $A$. The number of walks of length $\ell\in\NN$ in $\G$ from $x$ to $y$ is equal to $(x,y)$-entry of the matrix $A^\ell$.
\end{lemma}

\begin{proof}
Pick $x,y\in X$. We prove the claim by mathematical induction on $h=\partial(x,y)$ (the distance from $x$ to $y$).

{\sc Base of induction.} For $\ell=1$ the claim is trivial.

{\sc Induction step.} Assume that, if $\partial(x,y)=h\in\{1,\ldots,m-1\}$ $(m\ge 2)$, then the number of walks of length $h$ is equal to $(x,y)$-entry of the matrix $A^h$. We prove that the claim is true for $m$.
\begin{align}
(A^m)_{xy} &= (A^{m-1}A)_{xy}\nonumber\\
&=\sum_{z\in X} (A^{m-1})_{xz}\cdot (A)_{zy}\label{fD}.
\end{align}
By the induction assumption, $(A^{m-1})_{xz}$ is the number of walks of length $m-1$ from $x$ to $z$. The entry $(A)_{zy}$ is nonzero (i.e., equal to the number of walks of length $1$ from $z$ to $y$) if and only if in $\G$ there is at least one arc from $z$ to $y$. So, in \eqref{fD}, we start from the sum $S=0$, and for every $z\in X$ if there is at least one arc $(z,y)$ we add the number $(A^{m-1})_{xz}\cdot (A)_{zy}$ to $S$. Thus, the sum \eqref{fD} represents the number of walks of length $m$ from $x$ to $y$, and the result follows.
\end{proof}

\begin{lemma}
\label{hC}
Let $B\in\Mat_X(\RR)$ denote a nonnegative matrix and $\G=\G(A)$ denote underlying digraph of $B$ with adjacency matrix $A$. Then, for any $\ell\in\NN$,
$$
(B^\ell)_{zy}\ne 0
\qquad\mbox{ if and only if }\qquad
(A^\ell)_{zy}\ne 0
\qquad(y,z\in X).
$$
\end{lemma}

\begin{proof}
Immediate from the definition of the underlying digraph and the underlying weighted digraph of~$B$.
\end{proof}


\subsection{A vector space of all polynomials in a normal nonnegative matrix}

Let $B\in\Mat_X(\RR)$ denote a normal nonnegative matrix. By Theorem~\ref{gd}, $B$ has $|X|$ linearly independent eigenvectors $\U=\{u_1,u_2,...,u_{|X|}\}$ which form an orthonormal basis for $\mathbb C^{|X|}$. Let $V_i$ denote the eigenspace $V_i=\ker(B-\lambda_i I)$ and $\dim(V_i)=m_i,$ for $0\le i\le d$. For every vector $u_i\in {\U}$ there exists exactly one eigenspace $V_j$ such that $u_i\in V_j$, and since $V_i\cap V_j=\{\0\}$ for $i\not=j$, we can partition the set $\U$ into sets $\U_0,$ $\U_1,$  ..., $\U_d$ such that 
$$
\U_i \mbox{ is a basis for } V_i,\qquad
\U=\U_0\cup \U_1 \cup \ldots \cup \U_d
\qquad\mbox{ and }\qquad 
\U_i\cap \U_j=\emptyset.
$$
Note that
$$
\CC^{|X|}=V_0\oplus V_1\oplus \cdots\oplus V_d\quad(\mbox{orthogonal direct sum})
$$
and
\begin{equation}
m_0+m_1+\cdots+m_d=|X|.
\end{equation}

\begin{definition}{\rm
\label{r0}
Let $B\in\Mat_X(\RR)$ denote a normal nonnegative matrix. For each eigenvalue $\lambda_i$ $(0\le i\le d)$ of $B$, let $U_i$ denote the matrix whose columns form an orthonormal basis of its eigenspace $V_i=\ker(B-\lambda_i I)\subseteq\CC^{|X|}$. The {\it primitive idempotents} of $B$ are matrices {$E_i$} defined in the following way:
$$
E_i := U_i\ol{U_i}^\top\qquad(0\le i\le d).
$$
}\end{definition}



\begin{proposition}
\label{ca}
With reference to {\rm Definition \ref{r0}}, let $\B=\{p(B)\mid p\in\CC[x]\}$ denote the vector space over $\CC$ of all polynomials in $B$. Then, the following hold.
\begin{enumerate}[label=\rm(\roman*)]
\item Any power of $B$ can be expressed as a linear combination of the idempotents $E_i$ $(0\le i\le d)$, i.e.,
$$
B^h=\sum_{i=0}^d \lambda_i^h E_i
\qquad
(h\in\NN).
$$
\item $\{E_0,E_1,\ldots,E_d\}$ is an orthogonal basis of $\B$.
\item $\{I,B,B^2,\ldots,B^d\}$ is a basis of $\B$.
\item $\ol{B}^\top= B^\top = p(B)$ for some polynomial $p\in\CC[t]$.
\end{enumerate}
\end{proposition}

\begin{proof}
(i) With respect to Definition~\ref{r0}, abbreviate $m_i=m(\lambda_i)$ $(0\le i\le d)$. Pick $i$ $(0\le i\le d)$, and note that
$
BU_i=
\lambda_iU_i.
$
If $U=[U_1|U_2|\ldots|U_d]$, then
$$
B=U\Lambda\ol{U}^{\top},
\qquad
\mbox{ where }
\qquad
\Lambda=\left[
\begin{matrix} 
\lambda_{0} I_{m_0} & 0 & \ldots & 0 \\
0 & \lambda_{1} I_{m_1} & \ldots & 0 \\
\vdots & \vdots & \ddots & \vdots \\
0 & 0& \ldots & \lambda_{d} I_{m_d}
\end{matrix}
\right].
$$
Now, it is routine to see that
\begin{align*}
B&=U\Lambda\ol{U}^{\top}=
[U_0|U_1|\ldots|U_d]
\left[
\begin{matrix} 
\lambda_0I_{m_0} & 0 & \cdots & 0 \\
0 & \lambda_1I_{m_1} & \cdots & 0 \\
\vdots & \vdots & \ddots & \vdots \\
0 & 0 & \ldots & \lambda_dI_{m_d} \\
\end{matrix}
\right]
\left[
\begin{matrix} 
\underline{\ol{U_0}^{\top}}\\
\underline{\ol{U_1}^{\top}}\\
\underline{\,\,\,\vdots\,\,\,}\\
\ol{U_d}^{\top}\\
\end{matrix}
\right]\\
&=[\lambda_0U_0|\lambda_1U_1|\cdots|\lambda_dU_d]
\left[
\begin{matrix} 
\underline{\ol{U_0}^{\top}}\\
\underline{\ol{U_1}^{\top}}\\
\underline{\,\,\,\vdots\,\,\,}\\
\ol{U_d}^{\top}\\
\end{matrix}
\right]\\
&=\lambda_0U_0\ol{U_0}^{\top}+\lambda_1 U_1\ol{U_1}^{\top}+\cdots+ \lambda_dU_d\ol{U_d}^{\top}\\
&=\underbrace{\lambda_0}_{\in\CC}E_0+\underbrace{\lambda_1}_{\in\CC}E_1+\cdots+\underbrace{\lambda_d}_{\in\CC} E_d\qquad(\mbox{where }E_i:=U_i\ol{U_i}^{\top}).
\end{align*}
Since $E_iE_j=\delta_{ij}E_i$, $\ol{E_i}^\top=E_i$ and $\trace(E_i)=m_i$ $(0\le i\le d)$, for any polynomial $p\in\CC[t]$ we have
\begin{align*}
p(B) &= \underbrace{p(\lambda_0)}_{\in\CC} E_0 + \cdots + \underbrace{p(\lambda_d)}_{\in\CC} E_d.
\end{align*}
The result follows.

(ii)--(iv) Routine. For (ii) and (iii), see, for example, \cite[Chapter~2]{SP}. For claim (iv), see, for example, \cite[Theorem~1]{CFGM}, or note that $\ol{E_i}^\top=E_i$.
\end{proof}


\subsection{Commutative association schemes}

In Section~\ref{1A}, we have already provided the definition of a commutative $d$-class association scheme $\XXi=\{X,\{R_i\}_{i=0}^d\}$ together with definitions of relations $\{R_i\}_{i=0}^d$, relation matrices $\{B_i\}_{i=0}^d$ and Bose--Mesner algebra $\M$ of $\XXi$. The meaning of a ``matrix generates $\XXi$'' has been given in Section~\ref{1A} as well. Note that relation matrices $\{B_i\}_{i=0}^d$ form a standard basis of the Bose-Mesner algebra $\M$. We say that the relation $R_i$ {\em generates} the association scheme $\XXi$ if every element of the Bose--Mesner algebra $\M$ of $\XXi$ can be writen as a polynomial in $B_i$, where $B_i$ is adjacency matrix of the (di)graph $(X,R_i)$. We say that the association scheme $\XXi$ is {\em $P$-polynomial} with respect to $B_1$, if it is generated by $B_1$ and there exists an ordering $(B_0,B_1,\ldots,B_d)$ and polynomials $p_j(t)$ of degree $j$, such that $B_j=p_j(B_1)$ $(0\le j\le d)$. It is well known that if $(B_0,B_1,\ldots,B_d)$ is a $P$-polynomial ordering of $\XXi$ then $\G=\G(B_1)$ is a distance-regular (di)graph. Recall that the association scheme $\XXi$ is {\em polynomial} (with respect to $R_i$) if it is generated by a relation $R_i$ for some $i$ (we recommend articles \cite{MPc,TtYl} for interesting results in that direction).


\section{The Hoffman polynomial of a nonnegative matrix}
\label{ra}

In this section, we prove Theorem~\ref{Qe}. Our proof is within the lines of \cite[Theorem~1]{HMc}. Theorem~\ref{Qe} states: For a nonnegative matrix $B\in\Mat_X(\RR)$ there exists a polynomial $p\in\CC[t]$ such that $p(B) = J$ if and only if $B$ is a $\lambda$-doubly stochastic irreducible matrix. Furthermore, it states that the unique polynomial of smallest degree satisfying $p(B)=J$ is $h(t)=\frac{|X|}{q(\lambda)}q(t)$, where $q(\lambda)\ne0$ and $(t-\lambda)q(t)$ is the minimal polynomial of $B$.

Hoffman polynomial is well-known polynomial in algebraic graph theory and can be found in many textbooks (see, for example, \cite{B,BB,R,C,L,W}). There has been a great deal of work following this concept, for example: polynomials that sends a nonnegative irreducible matrix to a positive rank one matrix \cite{WD}, some Hoffman-type identities for the class of harmonic and semiharmonic graphs \cite{D}, Hoffman identities of non-regular graphs through the use of the Laplacian \cite{Y}, Hoffman identities by means of main eigenvalues \cite{T}, Hoffman polynomial of the tensor product of a cycle \cite{F}, Hoffman polynomial of cycle prefix digraphs \cite{M}, Hoffmn polynomials of some more general regular strongly connected digraphs \cite{A}.

\bigskip
\noindent
{\bf Proof of Theorem~\ref{Qe}.} $(\La)$ Assume that $B$ is a $\lambda$-doubly stochastic irreducible matrix. We use this assumption to show that there exists a polynomial $p\in\CC[t]$ such that $p(B) = J$. Note that, by assumption, $B\jj=\lambda\jj=B^\top\jj$. 

A square matrix is {\em stochastic} if all of its entries are nonnegative, and the entries of each column sum to $1$. From, for example, \cite[Subsection~5.6]{mR}, if $M$ is a stochastic matrix, then $1$ is an eigenvalue of $M$; and if $\theta$ is a (real or complex) eigenvalue of $M$, then $|\theta|\le 1$. In our case, we have that $\lambda^{-1}B$ is a stochastic matrix. It is routine to show that $\eig(B)=\{\lambda,\lambda_1,\ldots,\lambda_d\}$ are eigenvalues of $B$ if and only if $\lambda^{-1}\eig(B)$ are eigenvalues of $\lambda^{-1}B$. It follows that $\lambda=\max_{\theta\in\eig(B)}|\theta|$. By Theorem~\ref{gC}, the algebraic multiplicity of $\lambda$ is equal to $1$ and consequently 
\begin{equation}
\label{gt}
\mbox{
the geometric multiplicity of $\lambda$ is equal to $1$}
\end{equation}
also. This implies that the minimal polynomial of $B$ is in the following form $m(t)=(t-\lambda)s(t)$ for some polynomial $s(t)\in\CC[t]$, where $s(\lambda)\ne 0$. Since $(B-\lambda I)s(B)=\OO$, for any $v\in\CC^{|X|}$ we have $(B-\lambda I)s(B)v=\0$, which yields $s(B)v\in\ker(B-\lambda I)$. Thus, 
\begin{equation}
\label{rb}
\mbox{for every $v\in\CC^{|X|}$ there exists $\alpha_{v}\in\CC$ such that $s(B)v=\alpha_v\jj$}.
\end{equation}

Next we consider when equation \eqref{rb} is possible. For the moment, assume that $\langle v,\jj \rangle=0$, where $\langle\cdot,\cdot\rangle$ stands for standard Hermitian inner product $\langle u,v \rangle=u\ol{v}^\top$. For every $\ell\in\NN$,
\begin{align*}
\langle B^\ell v, \jj \rangle &= \langle v, \ol{B^\ell}^\top\jj \rangle\\
&= \lambda^\ell \langle v,\jj \rangle\\
&= 0.
\end{align*}
This yields that $\langle s(B)v,j\rangle=0$, and by \eqref{rb}, $\langle \alpha_v\jj,\jj\rangle=0=\alpha_v\|\jj\|^2$, which implies $\alpha_v=0$. Thus, for every vector $v\in\CC^{|X|}$ for which $\langle v,\jj \rangle=0$, we also have $s(B)v=\0$. Since $s(B)\jj=s(\lambda)\jj$ (as well as $J\jj=|X|\jj$), and $\CC^{|X|}=\langle\jj\rangle\oplus \langle\jj\rangle^\bot$ (orthogonal direct sum), we can conclude that the polynomial
$$
h(t)=\frac{|X|}{s(\lambda)} s(t)
$$
has the property that $h(B)=J$.

\bigskip
$(\Ra)$ Assume now that for a nonnegative matrix $B\in\Mat_X(\RR)$, there exists a polynomial $p\in\CC[t]$ such that $p(B) = J$. We use this assumption to show that $B$ is a $\lambda$-doubly stochastic irreducible matrix, for some $\lambda\in\RR$.

Let $y,z\in X$ denote two arbitrary elements of $X$, and $A$ denote the adjacency matrix of the underlying digraph of $B$. Since $p(B) = J$, there exists $\ell$ such that $(B^\ell)_{zy}\ne 0$, which is true if and only if $(A^\ell)_{zy}\ne 0$ (see Lemma~\ref{hC}). Thus the digraph $\G=\G(A)$ (digraph with adjacency matrix $A$) is strongly connected. Note that $\G$ is strongly connected if and only if $A$ is an irreducible matrix (see Lemma~\ref{gb}). Therefore, our non-negative matrix $B$ is also irreducible (see Lemma~\ref{gr}).

To prove that $B$ is $\lambda$-doubly stochastic matrix, for some $\lambda\in\RR$, we pick $x\in X$, and we consider the out-edge-weight-sum $\Sigma_1^\ra(x)$ of the vertex $x$, as well as the in-edge-weight-sum $\Sigma_1^\la(x)$ of the same vertex:
$$
\Sigma_1^\ra(x) = \sum_{z\in X} (B)_{xz} = (B\jj)_x,
$$
$$
\Sigma_1^\la(x) = \sum_{z\in X} (B)_{zx} = (B^\top\jj)_x.
$$
Since $J=p(B)$, we have $JB=BJ$, and with it, for any $y\in X$, we have
\begin{align*}
\theta &= \Sigma_1^\ra(x)
= (B\jj)_x\\
&= \sum_{z\in X} (B)_{xz} (J)_{zy}
= (BJ)_{xy}\\
&= (JB)_{xy}
= \sum_{z\in X} (J)_{xz} (B)_{zy}\\
&= \sum_{z\in X} (B^\top)_{yz} (J)_{zx}
= (B^\top\jj)_y\\
&= \Sigma_1^\la(y).
\end{align*}
From above, we have that every $y\in X$ (including our fixed $x$) has the same in-edge-weight-sum, i.e., $\Sigma_1^\la(y)=\theta$ $(\forall y\in X)$. Again, considering $JB=BJ$, for any $y\in X$, we have
\begin{align*}
\theta &= \Sigma_1^\la(x)
= (B^\top\jj)_x\\
&= \sum_{z\in X} (B^\top)_{xz} (J)_{zy}
= \sum_{z\in X} (J)_{yz} (B)_{zx}\\
&= (JB)_{yx}
= (BJ)_{yx}\\
&= \sum_{z\in X} (B)_{yz} (J)_{zx}
= (B\jj)_y\\
&= \Sigma_1^\ra(y).
\end{align*}
Thus, every $y\in X$ has the same out-edge-weight-sum $\theta$ also, i.e., $\Sigma_1^\ra(y)=\theta$ $(\forall y\in X)$. This conclude the proof that $B$ is a $\lambda$-doubly stochastic irreducible matrix, for some $\lambda\in\RR$.

\bigskip
It is left to prove that the unique polynomial of smallest degree satisfying $p(B)=J$ is $h(t)=\frac{|X|}{q(\lambda)}q(t)$ where $q(\lambda)\ne 0$ and $(t-\lambda)q(t)$ is the minimal polynomial of $B$.

We had already shown that $h(B)=J$ for $h(t)=\frac{|X|}{q(\lambda)}q(t)$, where $q(\lambda)\ne0$ and $m(t)=(t-\lambda)q(t)$ is the minimal polynomial of $B$. Assume that degree of $q(t)$ is $\ell$, and note that $\ell$ is smaller than the degree of the minimal polynomial $m(t)$ of $B$. We first prove that $h(t)$ is the unique polynomial of degree $\ell$ such that $h(B)=J$; our proof is by a contradiction. Assume that $r(t)$ is a polynomial of degree $m$ (in particular for this case $m=\ell$) such that $r(B)=J$, and that $r(t)\ne h(t)$. We have $(r-h)(B)=\OO$, and since degree of $r(t)-h(t)$ is less or equal to $\ell$, this is possible if and only if $r(t)-h(t)=0$, a contradiction. Thus, $h(t)$ is the unique such polynomial of degree $\ell$. In a similar way we show that there are no polynomials $r(t)$ of degree smaller than $\ell$ that satisfy $r(B)=J$. The result follows.

\begin{corollary}
Assume that $B\in\Mat_{X}(\RR)$ is a normal $\lambda$-doubly stochastic irreducible matrix, and let $\eig(B)=\{\lambda,\lambda_1,\ldots,\lambda_d\}$ denote the set of distinct eigenvalues of $B$. Then, the Hoffman polynomial of $B$ is
$$
h(t)=\frac{|X|}{\pi_0} \prod_{i=1}^d (t-\lambda_i),
$$
where $\pi_0=\prod_{i=1}^d (\lambda-\lambda_i)$.
\end{corollary}

\begin{proof}
Define $\theta=\max\limits_{\mu\in\eig(B)}|\mu|$. Since $B$ is a nonnegative irreducible matrix, by Theorem~\ref{gC}, $\theta\in\eig(B)$, $\theta>0$, the algebraic multiplicity of $\theta$ is $1$, and there exists an eigenvector $\nnu$ with all positive entries (normalized in such a way that the smallest entry is equal to $1$), such that $B\nnu = \theta \nnu$. On the other hand, since $B$ is $\lambda$-doubly stochastic, we have $B\jj=\lambda\jj$. By definition of $\theta$ we have $\lambda\le\theta$. Next, we prove that $\theta\le\lambda$. Let $\nnu=(\nu_x,\ldots,\nu_w)^\top$ and define $\nu_y=\max\limits_{z\in X}\{\nu_z\}$. We have
$$
\theta\nu_y=(\theta \nnu)_y=(B\nnu)_y=\sum_{z\in X} (B)_{yz} \nu_z\le \nu_y \sum_{z\in X} (B)_{yz} =\nu_y\lambda.
$$
As consequence, $\theta=\lambda$ and $\nnu=\jj$.

From the above observation, now we can let $m(t)=(t-\lambda)q(t)$ denote the minimal polynomial of $B$. In the end, since $B$ is a normal matrix, $m(t)=(t-\lambda) \prod_{i=1}^d (t-\lambda_i)$ (see, for example, \cite[Section~7.11]{MCm}). The result follows from Theorem~\ref{Qe}.
\end{proof}



\section{On predistance polynomials}
\label{OA}

In this section, we define predistance polynomials, the set of orthogonal polynomials that we use for the rest of the paper. The term ``predistance polynomial'' is taken from the theory of distance-regular graphs (see, for example, \cite{FAb,FAc,FAe,FaD,FAf}).

We define an inner product on $\Mat_X(\CC)$ in the following way:
\begin{equation}
\label{ob}
\langle B, C\rangle=\frac{1}{|X|}\trace(B\ol{C}^\top)
\qquad(B, C\in \Mat_X(\CC)).
\end{equation}
Let 
$$
\|C\|^2=\langle C, C\rangle\qquad\mbox{ for all $C\in\Mat_X(\CC)$,}
$$
and note that for any $R,S\in\Mat_X(\CC)$,
\begin{equation}
\label{oc}
\langle R,S\rangle=
\frac{1}{|X|}\sum\limits_{u\in X}(R\ol{S}^\top)_{uu}
=\frac{1}{|X|}\sum\limits_{u\in X}\sum\limits_{v\in X}(R)_{uv}(\ol{S})_{uv}
=\frac{1}{|X|}\sum\limits_{u,v\in X}(R\circ \ol{S})_{uv},
\end{equation}
(where $\circ$ is the elementwise-Hadamard product).

Now, let $B\in\Mat_{X}(\CC)$ denote a normal matrix with $d+1$ distinct eigenvalues. Let $\mathbb{C}_d[t]=\{a_0+a_1t+\ldots+a_dt^d \mid a_i\in\mathbb{C},\,0\le i\le d\}$ denote the ring of all polynomials of degree at most $d$ with coefficients in $\mathbb{C}$. For every $p,q\in\CC_d[t]$ we define
\begin{equation}
\label{oe}
\langle p, q\rangle={1\over |X|}\trace(p(B)\ol{q(B)}^\top),
\end{equation}
and let $\|p\|^2=\langle p,p\rangle$. We also have that \eqref{oe} is an inner product in $\mathbb{C}_d[t]$.

\begin{lemma}
\label{Qi}
Let $B\in\Mat_{X}(\RR)$ denote a normal real matrix with $d+1$ distinct eigenvalues $\{\lambda,\lambda_1,\ldots,\lambda_d\}\subseteq\CC$ (real or complex) and let $\mathbb{R}_d[t]=\{a_0+a_1t+\ldots+a_dt^d \mid a_i\in\mathbb{R},\,0\le i\le d\}$ denote the ring of all polynomials of degree at most $d$ with coefficients in $\mathbb{R}$. The minimal polynomial $m(t)$ of $B$ is of degree $d+1$ and $m(t)\in\RR_d[t]$.
\end{lemma}

\begin{proof}
Since $B$ is diagonalizable (see Theorem~\ref{gd}), the minimal polynomial $m(t)$ of $B$ is
\begin{equation}
\label{Qh}
m(t)=(t-\lambda)(t-\lambda_1)\cdot\cdots\cdot(t-\lambda_d)
\end{equation}
(see, for example, \cite[Subchapter~7.11]{MCm}). As $B$ has real entries, its characteristic polynomial $c(t)=\det(B-tI)$ will only have real coefficients. The complex roots of $c(t)=0$ come in conjugate complex pairs, yielding $m(t)\in\RR_d[t]$ from \eqref{Qh}.
\end{proof}

\begin{lemma}
\label{Qp}
Let $B\in\Mat_{X}(\RR)$ denote a normal real matrix with $d+1$ distinct eigenvalues $\{\lambda,\lambda_1,\ldots,\lambda_d\}\subseteq\CC$ (real or complex), and let $\mathbb{R}_d[t]=\{a_0+a_1t+\ldots+a_dt^d \mid a_i\in\mathbb{R},\,0\le i\le d\}$ denote the ring of all polynomials of degree at most $d$ with coefficients in $\mathbb{R}$. For every $p,q\in\RR_d[t]$ we define an inner product $\langle p, q\rangle$ on $\RR_d[x]$ as in \eqref{oe}. If $\lambda\ne 0$, then there exists an orthogonal system of polynomials $\{q_0(t),q_1(t),\ldots,q_d(t)\}\subseteq\RR_d[t]$ such that every $q_i(t)$ $(0\le i\le d)$ has degree $i$ and $q_i(\lambda)\ne 0$ $(0\le i\le d)$ (i.e., $\lambda$ is not a root of $q_i(t)$ $(0\le i\le d)$).
\end{lemma}

\begin{proof}
Our proof is by construction. Using the inner product \eqref{oe}, we apply the Gram--Schmidt orthogonalization algorithm to the set $\{s_0(t)=1,s_1(t)=t,\ldots,s_d(t)=t^d\}$, modifying it in such a way to meet our conditions.

Note that, since $B\in\Mat_X(\RR)$ and $p,q\in\RR_d[t]$, by \eqref{oe} we have $\langle p, q\rangle\in\RR$. To construct the $q_i$'s, we use mathematical induction on $i$.

{\sc Base of induction.} Since $s_0(t)=1$ is a constant function, $\lambda$ is not a root of $s_0(t)$. So, we can define $q_0(t):=s_0(t)$. Next, define $r_1(t)$ in the following way
\begin{align*}
r_1(t) &:= s_1(t) - \sum_{\ell=0}^{0} \frac{\langle q_\ell,s_1\rangle}{\|q_\ell\|^2} q_\ell(t)\\
&=s_1(t) - \frac{\langle q_0,s_1\rangle}{\|q_0\|^2} q_0(t).
\end{align*}
For the polynomial $r_1(t)$ two cases are possible: either $r_1(\lambda)\ne 0$ or $r_1(\lambda)=0$.
In the first case (i.e., $r_1(\lambda)\ne 0$), we let $q_1(t):=r_1(t)$. 
In the second case (i.e., $r_1(\lambda)=0$), we have $\lambda=s_1(\lambda)= \frac{\langle q_0,s_1\rangle}{\|q_0\|^2} q_0(\lambda)$, which yields, for example, $2\lambda\ne  \frac{\langle q_0,s_1\rangle}{\|q_0\|^2} q_0(\lambda)$. Thus, we define $q_1(t)$ in the following way
$$
q_1(t) := 2s_1(t) - \frac{\langle q_0,s_1\rangle}{\|q_0\|^2} q_0(t).
$$
In both cases, $q_1(\lambda)\ne 0$ and $q_1(t)$ is a polynomial of degree $1$.

{\sc Induction step.} Assume that we found an orthogonal set of polynomials\break $\{q_0(t),q_1(t),\ldots,q_{j-1}(t)\}\subseteq\RR_d[t]$ such that $q_i(\lambda)\ne 0$ ($0\le i\le j-1$, $j\ge 2$), i.e., such that $\lambda$ is not a root of $q_i(t)$ (and assume that each $q_i$ has degree $i$). Now, we define $r_j(t)$ in the following way
\begin{align*}
r_j(t) &:= s_j(t) - \sum_{\ell=0}^{j-1} \frac{\langle q_\ell,s_j\rangle}{\|q_\ell\|^2} q_\ell(t)
\end{align*}
For the polynomial $r_j(t)$ two cases are possible: either $r_j(\lambda)\ne 0$ or $r_j(\lambda)=0$. 
In the first case (i.e., $r_j(\lambda)\ne 0$), we let $q_j(t):=r_j(t)$. 
In the second case (i.e., $r_j(\lambda)=0$), we have
$$
\lambda^j=s_j(\lambda)= \sum_{\ell=0}^{j-1} \frac{\langle q_\ell,s_j\rangle}{\|q_\ell\|^2} q_\ell(\lambda),
$$ which yields, for example, $2\lambda^j\ne  \sum_{\ell=0}^{j-1} \frac{\langle q_\ell,s_j\rangle}{\|q_\ell\|^2} q_\ell(\lambda)$. Thus, we define $q_j(t)$ in the following way
$$
q_j(t) := 2s_j(t) - \sum_{\ell=0}^{j-1} \frac{\langle q_\ell,s_j\rangle}{\|q_\ell\|^2} q_\ell(t)
$$
In both cases, $q_j(\lambda)\ne 0$ and $q_j(t)$ is a polynomial of degree $j$ (by construction).
\end{proof}

\begin{lemma}
\label{gu}
Let $B\in\Mat_{X}(\RR)$ denote a normal irreducible matrix with $d+1$ distinct eigenvalues $\{\lambda,\lambda_1,\ldots,\lambda_d\}$, such that $B\jj=\lambda\jj$ $(\lambda\ne 0)$. For every $p,q\in\RR_d[t]$ we define an inner product $\langle p, q\rangle$ on $\RR_d[x]$ as in \eqref{oe}, and let $\|p\|^2=\langle p,p\rangle$. Then, there exists a set of orthogonal polynomials with respect to \eqref{oe}, such that $\deg(p_i)=i$ $(0\le i\le d)$ and they are normalized in such a way that $\|p_i\|^2=p_i(\lambda)\in\RR^+$ $(0\le i\le d)$.
\end{lemma}

\begin{proof}Our proof is by construction. By Lemma~\ref{Qp}, we always can find an orthogonal system of polynomials $\{q_0(t),q_1(t),\ldots,q_d(t)\}\subseteq\RR_d[t]$ such that $q_i(\lambda)\ne 0$ $(0\le i\le d)$ (i.e., that $\lambda$ is not a root of any $q_i(t)$) and that each $q_i$ is of degree $i$. Next, we first define polynomial $r_i(t)$ $(0\le i\le d)$ on the following way
\begin{align*}
r_i(t) &= \frac{1}{\|q_i\|} q_i(t)\qquad (0\le i\le d).
\end{align*}
Note that  $\{r_0(t),r_1(t),\ldots,r_d(t)\}\subseteq\RR_d[t]$ orthonormal system of polynomials, and that we have $\dgr r_i=i$, $\|r_i\|=1$ $(0\le i\le d)$, and by our choice of the $q_i$'s, we also have $r_i(\lambda)\ne 0$ $(0\le i\le d)$. For arbitrary nonzero real numbers $\alpha_0,\alpha_1,\ldots,\alpha_d$, the set $\{\alpha_0r_0,\alpha_1r_1,\ldots,\alpha_dr_r\}$ is again an orthogonal set. For any $r_i(t)\in\RR_d[t]$ $(0\le i\le d)$, define $c:=r_i(\lambda)$ and $p_i(t):=cr_i(t)$ (note that $c\ne 0$ by our construction). We have
$$
\|p_i\|^2=\langle cr_i,cr_i\rangle=c^2\underbrace{\|r_i\|}_{=1}=c\cdot c=cr_i(\lambda)=p_i(\lambda).
$$
Thus, $\|p_i\|^2=p_i(\lambda)$ $(0\le i\le d)$. Note that the set $\{p_0(t),p_1(t),\ldots,p_d(t)\}$ is an orthogonal system and $p_0(t)=1$.
\end{proof}

\begin{definition}
\label{of}{\rm
Let $B\in\Mat_{X}(\RR)$ denote a normal irreducible matrix with $d+1$ distinct eigenvalues $\{\lambda,\lambda_1,\ldots,\lambda_d\}$, such that $B\jj=\lambda\jj$. For every $p,q\in\RR_d[t]$ we define $\langle p, q\rangle$ as in \eqref{oe} (i.e., $\langle p, q\rangle={1\over |X|}\trace(p(B)\ol{q(B)}^\top)$), and let $\|p\|^2=\langle p,p\rangle$. With reference to Lemma~\ref{gu}, the set of so-called {\em predistance polynomials} $\{p_0,p_1,\ldots,p_d\}\subseteq\RR_d[t]$, is a set of orthogonal polynomials with respect to the inner product \eqref{oe} (defined on the vector space $\RR_d[t]$), such that $\deg(p_i)=i$ $(0\le i\le d)$ and they are normalized in such a way that $\|p_i\|^2=p_i(\lambda)$. Note that $p_i(\lambda)\in\RR^+$ $(0\le i\le d)$.
}\end{definition}

\begin{lemma}
\label{oi}
With reference to {\rm Definition~\ref{of}}, let $B\in\Mat_X(\RR)$ denote a normal $\lambda$-doubly stochastic irreducible matrix with $d+1$ distinct eigenvalues. If $\{p_0,p_1,\ldots,p_d\}$ denotes the set of the predistance polynomials, then
$$
\sum_{i=0}^d p_i(B) = J.
$$
\end{lemma}

\begin{proof}
Our proof uses the same technique that implicitly can be found, for example, in \cite{CFfg,CTS,FGla,HoP}.

Let $h(t)$ denote the Hoffman polynomial, i.e., $h(B)=J$ (see Theorem~\ref{Qe}). If we denote by $\{\lambda,\lambda_1,\lambda_2,\ldots,\lambda_d\}$ the set of the distinct eigenvalues of $B$, then $h(\lambda)=|X|$ and $h(\mu)=0$ for $\mu\in\{\lambda_1,\lambda_2,...,\lambda_d\}$ (see Theorem~\ref{Qe}). Since $B$ is a normal matrix, there exists a unitary matrix $U$ such that $B=U\Lambda\ol{U}^\top$, where $\Lambda$ is diagonal matrix in which the diagonal entries are the eigenvalues of $B$. Let $\diag(\Lambda)$ denote the list of all diagonal entries of $\Lambda$. Then, we have
\begin{align*}
\langle h,p_j\rangle &= {1\over |X|}\trace(h(B)\ol{p_j(B)}^\top)\\
&={1\over |X|}\trace(Uh(\Lambda)\ol{p_j(\Lambda)}^\top \ol{U}^{\top})\\
&={1\over |X|}\trace(h(\Lambda)\ol{p_j(\Lambda)}^\top)\\
&={1\over |X|}\trace(h(\Lambda)\ol{p_j(\Lambda)})\\
&={1\over |X|}\sum_{\mu\in\diag(\Lambda)} h(\mu) \ol{p_j(\mu)}\\
&={1\over |X|}\cdot h(\lambda)\cdot{p_j(\lambda)}
\end{align*}
which yields
\begin{equation}
\label{oh}
\langle h,p_j\rangle=\|p_j\|^2 \qquad(0\le j\le d).
\end{equation}
On the other hand, the Fourier expansion of $h$ is
\begin{equation}
\label{og}
h=\frac{\langle h,p_0\rangle}{\|p_0\|^2}p_0+
\frac{\langle h,p_1\rangle}{\|p_1\|^2}p_1+
\cdots+
\frac{\langle h,p_d\rangle}{\|p_d\|^2}p_d.
\end{equation}
By \gzit{oh} and \gzit{og}, the result follows.
\end{proof}

\begin{proposition}
\label{ol}
With reference to {\rm Definition~\ref{of}}, let $B\in\Mat_X(\RR)$ denote a normal $\lambda$-doubly stochastic irreducible matrix with $d+1$ distinct eigenvalues. 
Let $\G=\G(A)$ denote the underlying digraph of $B$, with adjacency matrix $A$, diameter $D$, and let $A_D$ denote the distance-$D$ matrix of $\G$.
Assume that $d=D$, and let $\{p_0,p_1,\ldots,p_D\}$ denote the set of the predistance polynomials. If there exists a polynomial $q(t)\in\RR_{D}[t]$ such that $A_D=q(B)$, then 
$$
q(t)=p_D(t).
$$
\end{proposition}

\begin{proof} By our assumption, $B\jj=\lambda\jj$. We first show that $\|q\|^2=q(\lambda)$. Note that $A_D\jj=q(B)\jj=q(\lambda)\jj$, and from \eqref{oc}
\begin{align*}
\|q\|^2 &= \frac{1}{|X|} \trace(q(B)\ol{q(B)}^\top)\\
&= \frac{1}{|X|} \sum_{x,y\in X} (A_D)_{xy} = \frac{1}{|X|} \sum_{x\in X} (A_D\jj)_{x}\\
&= \frac{1}{|X|} \cdot |X|\cdot q(\lambda)\\
&= q(\lambda).
\end{align*}
Recall that $\{p_0,p_1,\ldots,p_D\}$ is a set of orthogonal polynomials such that $\deg(p_i)=i$ and $\|p_i\|^2=p_i(\lambda)$ $(0\le i\le D)$. Next note that
\begin{align*}
\langle q,p_i\rangle &= \frac{1}{|X|} \trace(q(B)\ol{p_i(B)}^\top)\\
&=\frac{1}{|X|} \sum_{x,y\in X} \left( A_D\circ \ol{p_i(B)} \right)_{xy}\\
&=\left\{
\begin{array}{ll}
0,&\mbox{ if $0\le i\le D-1$},\\
\frac{1}{|X|} \ds\sum_{x,y\in X} \left( A_D\circ \ol{p_D(B)} \right)_{xy},&\mbox{ if $i=D$}
\quad
\end{array}
\right. (\mbox{by Lemmas~\ref{hB}, \ref{hC}}).
\end{align*}
This yields that the Fourier expansion of $q=\sum_{i=0}^D \frac{\langle q,p_i\rangle}{\|p_i\|^2} p_i$ is equal to
$$
q=\frac{\langle q,p_D\rangle}{\|p_D\|^2} p_D
\qquad\mbox{and}\qquad \langle q,p_D\rangle\ne 0
$$
which implies $p_D=c\cdot q$ where $c=\frac{\|p_D\|^2}{\langle q,p_D\rangle}$. To conclude, we show that $c=1$:
\begin{align*}
q(\lambda) &= \frac{\langle q,p_D\rangle}{\|p_D\|^2} \underbrace{p_D(\lambda)}_{=\|p_D\|^2}\\
&= \langle q,p_D\rangle
= \langle q,c q\rangle\\
&= \ol{c}\langle q,q\rangle
= \ol{c}\|q\|^2\\
&= \ol{c} q(\lambda).
\end{align*}
The result follows.
\end{proof}

\begin{lemma}
\label{Qj}
With reference to {\rm Definition~\ref{of}}, let $B\in\Mat_X(\RR)$ denote a normal $\lambda$-doubly stochastic irreducible matrix with $d+1$ distinct eigenvalues. 
Let $\G=\G(A)$ denote the underlying digraph of $B$ with diameter $D$, and let $A_D$ denote the distance-$D$ matrix of $\G$. Assume that $D\ge 3$. For any $x,y\in X$, if $\partial(x,y)\le D-2$ in $\G$, then
$$
(A_DB^\top)_{xy} = 0.
$$
\end{lemma}

\begin{proof}
For any $x,y\in X$, we have
\begin{align*}
(A_DB^\top)_{xy} &= \sum_{z\in X} (A_D)_{xz}(B^\top)_{zy}\\
&= \sum_{z\in\G_D^\ra(x)} (B)_{yz}.
\end{align*}
Our proof is by a contradiction. Assume that there exists $x,y\in X$ such that $\partial(x,y)\le D-2$, and $(A_DB^\top)_{xy}\ne 0$. This yields $\sum_{z\in\G_D^\ra(x)} (B)_{yz}\ne 0$, i.e., there exists $z\in\G^\ra_D(x)$ such that $(B)_{yz}\ne 0$, or equivalently $(A)_{yz}=1$. Now consider the distance-$i$ partition $\{\G^{\ra}_i(x)\}_{i=0}^D$ of the vertex set $X$ (for our choice of $x\in X$). Since $\partial(x,y)\le D-2$ and $(A)_{yz}=1$, it follows that $\partial(x,z)\le D-1$, a contradiction with $z\in\G^\ra_D(x)$. The result follows.
\end{proof}


\section{Case when $\boldsymbol{A_D}$ is polynomial in $\boldsymbol{B}$}
\label{oa}

In this section, we prove Theorem~\ref{Qg}. For this purpose, we need Proposition~\ref{La} and Lemma~\ref{Ld}.


\begin{proposition}
\label{La}
With reference to {\rm Definition~\ref{of}}, let $B\in\Mat_X(\RR)$ denote a normal $\lambda$-doubly stochastic irreducible matrix with $d+1$ distinct eigenvalues. 
Let $\G=\G(A)$ denote the underlying digraph of $B$, with adjacency matrix $A$, diameter $D$, and let $\{A_0,A_1,\ldots,A_D\}$ denote the distance-$i$ matrices of $\G$.
Assume that $d=D$, and let $\{p_0,p_1,\ldots,p_D\}$ denote the set of predistance polynomials. 
If $A_D=p_D(B)$, $A_{D-1}=p_{D-1}(B),\ldots,A_{i+1}=p_{i+1}(B)$, and if there exists a polynomial $q(t)\in\RR_{i}[t]$ such that $A_i=q(B)$, then 
$$
q(t)=p_i(t).
$$
\end{proposition}

\begin{proof} The proof is similar to the proof of Proposition~\ref{ol}. By assumption $B\jj=\lambda\jj$. We first show that $\|q\|^2=q(\lambda)$. Note that $A_i\jj=q(B)\jj=q(\lambda)\jj$, and from \eqref{oc}
\begin{align*}
\|q\|^2 &= \frac{1}{|X|} \trace(q(B)\ol{q(B)}^\top)\\
&= \frac{1}{|X|} \sum_{x,y\in X} (A_i)_{xy}\\
&= \frac{1}{|X|} \cdot |X|\cdot q(\lambda)\\
&= q(\lambda).
\end{align*}
Recall that $\{p_0,p_1,\ldots,p_D\}$ is a set of orthogonal polynomials such that $\deg(p_j)=j$ and $\|p_j\|^2=p_j(\lambda)$ $(0\le j\le D)$. Next note that
\begin{align*}
\langle q,p_j\rangle &= \frac{1}{|X|} \trace(q(B)\ol{p_j(B)}^\top)\\
&=\frac{1}{|X|} \sum_{x,y\in X} \left( A_i\circ \ol{p_j(B)} \right)_{xy}\\
&=\frac{1}{|X|} \sum_{x,y\in X} \left( A_i\circ {p_j(B)} \right)_{xy}\\
&=\left\{
\begin{array}{ll}
0,&\mbox{ if $0\le j\le i-1$},\\
\ds\frac{1}{|X|} \ds\sum_{x,y\in X} \left( A_i\circ \ol{p_i(B)} \right)_{xy},&\mbox{ if $j=i$,}\\
0,&\mbox{ if $i+1\le j\le D$}\\
\end{array}
\right. (\mbox{by Lemmas~\ref{hB}, \ref{hC}}).
\end{align*}
This yields that the Fourier expansion of $q=\sum_{j=0}^D \frac{\langle q,p_j\rangle}{\|p_j\|^2} p_j$ is equal to
$$
q=\frac{\langle q,p_i\rangle}{\|p_i\|^2} p_i
\qquad\mbox{and}\qquad \langle q,p_i\rangle\ne 0
$$
which implies $p_i=c\cdot q$ where $c=\frac{\|p_i\|^2}{\langle q,p_i\rangle}$. To conclude, we show that $c=1$:
\begin{align*}
q(\lambda) &= \frac{\langle q,p_i\rangle}{\|p_i\|^2} \underbrace{p_i(\lambda)}_{=\|p_i\|^2}\\
&= \langle q,p_i\rangle
= \langle q,c q\rangle\\
&= \ol{c}\langle q,q\rangle
= \ol{c}\|q\|^2\\
&= \ol{c} q(\lambda).
\end{align*}
The result follows.
\end{proof}



\begin{lemma}
\label{Ld}
With reference to {\rm Definition~\ref{of}}, let $B\in\Mat_X(\RR)$ denote a normal $\lambda$-doubly stochastic irreducible matrix with $d+1$ distinct eigenvalues. 
Let $\G=\G(A)$ denote the underlying digraph of $B$ with diameter $D$, and let $\{A_0,A_1,\ldots,A_D\}$ denote the distance-$i$ matrices of $\G$. Assume that $D-j-1\ge 2$. For any $x,y\in X$, if $\partial(x,y)< D-j-1$ in $\G$ then
$$
(A_{D-j}B^\top)_{xy} = 0.
$$
\end{lemma}

\begin{proof}
The proof is similar to the proof of Lemma~\ref{Qj}. For any $x,y\in X$, we have
\begin{align*}
(A_{D-j}B^\top)_{xy} &= \sum_{z\in X} (A_{D-j})_{xz}(B^\top)_{zy}\\
&= \sum_{z\in\G_{D-j}^\ra(x)} (B)_{yz}.
\end{align*}

Our proof is by a contradiction. Assume that there exists $x,y\in X$ such that $\partial(x,y)\le D-j-2$, and $(A_{D-j}B^\top)_{xy}\ne 0$. This yields $\sum_{z\in\G_{D-j}^\ra(x)} (B)_{yz}\ne 0$, i.e., there exists $z\in\G^\ra_{D-j}(x)$ such that $(B)_{yz}\ne 0$, or equivalently $(A)_{yz}=1$. Now consider the distance-$i$ partition $\{\G^{\ra}_i(x)\}_{i=0}^D$ of the vertex set $X$ (for our choice of $x\in X$). Since $\partial(x,y)\le D-j-2$ and $(A)_{yz}=1$, we have $\partial(x,z)\le D-j-1$, a contradiction with $z\in\G^\ra_{D-j}(x)$. The result follows.
\end{proof}


\begin{proposition}
\label{ok}
With reference to {\rm Definition~\ref{of}}, let $B\in\Mat_X(\RR)$ denote a normal $\lambda$-doubly stochastic irreducible matrix with $d+1$ distinct eigenvalues. 
Let $\G=\G(A)$ denote the underlying digraph of $B$, with adjacency matrix $A$, diameter $D$, and let $A_D$ denote the distance-$D$ matrix of $\G$.
Assume that $d=D$, and let $\{p_0,p_1,\ldots,p_D\}$ denote the set of predistance polynomials. 
If $A_D$ is a polynomial in $B$, then
$$
A_i=p_i(B)\qquad(0\le i\le D).
$$
\end{proposition}

\begin{proof}
For the moment, let $q_m(t)=\sum_{i=0}^m p_i(t)$ $(0\le m\le D)$, and note that $\deg(q_m)=m$ $(0\le m\le D)$. Since $A_D$ is a polynomial in $B$, by Proposition~\ref{ol}, $p_D(B)=A_D$.

We first prove that $A_{D-1}=p_{D-1}(B)$. If $D=2$, the result follows (because $p_0(t)=1$, by assumption $p_2(B)=A_2$, and $J=\sum_{i=0}^2 p_i(B)$ yields $p_1(B)=A$). Assume that $D\ge 3$. Pick $x,y\in X$ such that $\partial(x,y)=D-1$, and note that $(q_{D-2}(B))_{xy}=0$. By Lemma~\ref{oi}, we have
\begin{align}
1=(J)_{xy}&=(q_D(B))_{xy}\nonumber\\
&=\left(q_{D-2}(B)+p_{D-1}(B)+p_D(B)\right)_{xy}\nonumber\\
&=\underbrace{(q_{D-2}(B))_{xy}}_{=0}+(p_{D-1}(B))_{xy}+\underbrace{(A_D)_{xy}}_{=0}\nonumber\\
&=(p_{D-1}(B))_{xy}.\label{OX}
\end{align}
We can conclude that, 
\begin{equation}
\label{ox}
\mbox{if $\partial(x,y)=D-1$, then $(p_{D-1}(B))_{xy}=1$.}
\end{equation}
Using \eqref{ox}, we show that $A_{D-1}$ is a polynomial in $B$.

Since $A_D$ is a polynomial in $B$, by Proposition~\ref{ol}, $p_D(B)=A_D$. By Lemma~\ref{Qj}, for any $x,y\in X$,
\begin{equation}
\label{oo}
(A_DB^\top)_{xy} = 
\left\{
\begin{array}{ll}
0,&\mbox{if }\partial(x,y)\le D-2,\\
\sum_{z\in\G_D^\ra(x)} (B)_{yz},&\mbox{if }\partial(x,y)\in\{D-1,D\}
\end{array}
\right.\quad (x,y\in X).
\end{equation}
Let $\B$ denote the vector space of all polynomials in $B$. Since $\{p_i(B)\}_{i=0}^D$ is a basis of $\B$, and $B^\top\in\B$ (see Proposition~\ref{ca}(iv)), there exist complex scalars $\alpha_h$ $(0\le h\le D)$ such that
\begin{equation}
\label{Qm}
A_DB^{\top} = \sum_{h=0}^D \alpha_h p_h(B)=\sum_{h=0}^{D-1} \alpha_h p_h(B) + \alpha_D \underbrace{p_D(B)}_{A_D}.
\end{equation}
By \eqref{oo} and \eqref{Qm}, note that
\begin{equation}
\label{Lb}
\left( \sum_{h=0}^D \alpha_h p_h(B) \right)_{xy}=0
\qquad\mbox{if } \partial(x,y)\le D-2.
\end{equation}
Now, from \eqref{OX}, \eqref{ox} and \eqref{Lb}, we have
$$
\left( \sum_{h=0}^{D-1} \alpha_h p_h(B) + \alpha_D A_D \right)_{xy}
=
\left\{
\begin{array}{ll}
0,&\mbox{if }\partial(x,y)\le D-2,\\
\alpha_{D-1},&\mbox{if }\partial(x,y)=D-1\\
\alpha_{D},&\mbox{if }\partial(x,y)=D
\end{array}
\right.\quad (x,y\in X),
$$
or, in other words,
\begin{equation}
\label{Lc}
\sum_{h=0}^{D-1} \alpha_h p_h(B) = \alpha_{D-1} A_{D-1}.
\end{equation}
If $\alpha_{D-1}=0$ then by \eqref{Qm} and \eqref{Lc}, $A_DB^\top=\alpha_D A_D$, a contradiction (see \eqref{oo}). Thus \eqref{Lc} yields that $A_{D-1}$ is a polynomial in $B$. By Proposition~\ref{La}, $p_{D-1}(B)=A_{D-1}$.

If $D=3$, the result follows. Assume that $D\ge 4$, and that we executed $j\ge 1$ steps from above (where $D-j-1\ge 1$), i.e., 
\begin{align*}
A_D&=p_D(B),\qquad A_{D-1}=p_{D-1}(B),\qquad \ldots,\qquad A_{D-j}=p_{D-j}(B).
\end{align*}
We now show that $A_{D-j-1}=p_{D-j-1}(B)$. Pick $x,y\in X$ such that $\partial(x,y)=D-j-1$, and note that $(q_{D-j-2}(B))_{xy}=0$ (by Lemma~\ref{hB}). By Lemma~\ref{oi},
\begin{align*}
1=(J)_{xy}&=(q_D(B))_{xy}\\
&=\left(q_{D-j-2}(B)+p_{D-j-1}(B)+\sum_{i=D-j}^D p_i(B)\right)_{xy}\\
&=\underbrace{(q_{D-j-2}(B))_{xy}}_{=0}+(p_{D-j-1}(B))_{xy}+\underbrace{\left(\sum_{i=D-j}^D p_i(B)\right)_{xy}}_{(A_{D-j}+\cdots+A_D)_{xy}=0}\\
&=(p_{D-j-1}(B))_{xy}.
\end{align*}
We can conclude that, 
\begin{equation}
\label{ov}
\mbox{if $\partial(x,y)=D-j-1$, then $(p_{D-j-1}(B))_{xy}=1$.}
\end{equation}
Using \eqref{ov}, we prove that $A_{D-j-1}$ is a polynomial in $B$.

Consider the product $A_{D-j}B^\top$. For arbitrary $x,y\in X$, by Lemma~\ref{Ld}, we have
\begin{equation}
\label{op}
(A_{D-j}B^\top)_{xy} 
= 0, \qquad\mbox{if }\partial(x,y)<D-j-1.
\end{equation}
Since $\{p_i(B)\}_{i=0}^D$ is a linearly independent set and $B^\top\in \B$, there exist complex scalars $\beta_h$ $(0\le h\le D)$ such that
\begin{equation}
\label{Le}
A_{D-j}B^{\top} = \sum_{h=0}^D \beta_h p_h(B) 
=\sum_{h=0}^{D-j-1} \beta_h p_h(B) + \beta_{D-j} A_{D-j} + \cdots + \beta_D A_D.
\end{equation}
By \eqref{ov}, \eqref{op} and \eqref{Le} we have
$$
\sum_{h=0}^{D-j-1} \beta_h p_h(B) = \beta_{D-j-1} A_{D-j-1}.
$$
Next we prove that $\beta_{D-j-1}\ne 0$. The proof is by contradiction. If $\beta_{D-j-1}= 0$, then \eqref{Le} becomes
$
A_{D-j}B^{\top} = \beta_{D-j} A_{D-j} + \beta_{D-j+1} A_{D-j+1} + \cdots +  \beta_D A_D.
$
This yields that $(A_{D-j}B^{\top})_{xy}=0$ for all $x\in X$ and $y\in\G^{\ra}_{D-j-1}(x)$, a contradiction (note that $(A_{D-j}B^{\top})_{xy}=\sum_{z_\in\G_{D-j}^{\ra}(x)} (B)_{yz}$). Thus, $A_{D-j-1}$ is a polynomial in $B$. By Proposition~\ref{La}, the result follows.
\end{proof}

In some sense, our Theorem~\ref{Qg} is similar to the following result from the theory of distance-regular graphs.

\begin{proposition}[{\cite[Proposition~2]{Fsp} or \cite{FGJ}}]
An undirected regular graph $\G=\G(A)$ with diameter $D$ and $d + 1$ distinct eigenvalues is distance-regular if and only if $D = d$ and the distance-$D$ matrix $A_D$ is a polynomial in $A$.
\end{proposition}

\subsection{Proof of Theorem~\ref{Qg}}
\label{QG}

In this subsection we prove Theorem~\ref{Qg}.


\medskip
$(\La)$ Assume that $B$ is a normal $\lambda$-doubly stochastic matrix with $D+1$ distinct eigenvalues and that $A_D$ is a polynomial in $B$. We use this assumption to show that $\B$ is the Bose--Mesner algebra of a commutative $D$-class association scheme.

Since $B$ is a normal matrix with $D+1$ distinct eigenvalues, by Proposition~\ref{ca}, $\{I, B,B^2,\ldots,\break B^D\}$ is a basis of $\B$. Since $A_D\in\B$, from Proposition~\ref{ok} it follows that the distance-$i$ matrices $A_i$ $(0\le i\le D)$ belong to $\B$. Furthermore since $\{A_i\}_{i=0}^D$ is a linearly independent set of matrices, it also forms a basis of $\B$. As $B$ is a normal matrix, note that $\ol{B}^\top\in\B$ by Proposition~\ref{ca}(iv). Now it is routine to check that the distance-$i$ matrices satisfy all properties (AS1)--(AS5) of a commutative association scheme.

\bigskip
$(\Ra)$ Assume that $\B$ is the Bose--Mesner algebra of a commutative $D$-class association scheme. We use this assumption to show that $B$ is a normal $\lambda$-doubly stochastic matrix (for some $\lambda$) with $D+1$ distinct eigenvalues and that the distance-$D$ matrix $A_D$ of $\G=\G(A)$ is a polynomial in $B$, where $A$ is the adjacency matrix of the underlying digraph of $B$.

The fact that $B$ belongs to a commutative association scheme implies that $B$ is a normal matrix. Since $B$ generates a commutative association scheme and $J$ belongs to the algebra of this scheme, by Theorem~\ref{Qe}, $B$ is a (normal) $\lambda$-doubly stochastic matrix (for some $\lambda$). For a moment, assume that $B$ has $d+1$ distinct eigenvalues; then, by Proposition~\ref{ca}, $\{I,B,B^2,\ldots,B^d\}$ is a basis of $\B$. This is possible only if $d=D$, and thus $B$ has $D+1$ distinct eigenvalues (see also, for example, \cite[Corollary~3.5]{MPc}).

It is left to show that the distance-$D$ matrix $A_D$ is polynomial in $B$. Let $\{B_0,B_1,\ldots,B_D\}$ denote the standard basis of $\B$ (the basis of $\circ$-idempotent $01$-matrices), and let $\theta_i$'s denote the nonzero scalars such that 
$$
B=\sum_{i\in\Phi} \theta_i B_i.
$$
for some nonempty set of indices $\Phi$. Note that, by definition, $A=\sum_{i\in\Phi} B_i$ is the adjacency matrix of the underlying digraph of $B$. Since $B$ is an irreducible matrix, the matrix $A$ is irreducible too (Lemma~\ref{gr}). Then, by Lemma~\ref{gb}, $\G$ is a strongly connected digraph, it is also a regular graph since $A\in\B$.

\medskip
We consider $\G=\G(A)$ and we finish the proof by proving Claims~1 and 2 below.

\medskip
\noindent
{\sc Claim 1.} For any $i$ $(0\le i\le d)$ and $y,z,u,v\in X$, if $(B_i)_{zy}=(B_i)_{uv}=1$, then $\partial(z,y)=\partial(u,v)$ in $\G$. 

\smallskip
{\bf Proof of Claim 1.} For every $\ell\in\NN$, there exist complex scalars $\alpha^{(\ell)}_i$ $(0\le i\le d)$ such that $A^\ell=\sum_{i=0}^{d} \alpha^{(\ell)}_i B_i$. Recall that $\sum_{i=0}^d B_i = J$ and $B_i\circ B_j=\delta_{ij}B_i$ $(0\le i,j\le d)$. This yields that, for any $y,z,u,v\in X$ and $i$ $(0\le i\le d)$, if $(B_i)_{zy}\ne 0$ and $(B_i)_{uv}\ne 0$, then $(A^\ell)_{zy}=(A^\ell)_{uv}=\alpha^{(\ell)}_i$, i.e., the number of walks of length $\ell$ from $z$ to $y$ is equal to the the number of walks of length $\ell$ from $u$ to $v$ (see Lemma~\ref{hB}). Moreover, $(A^\ell)_{zy}=(A^\ell)_{uv}$ holds for any $\ell$ $(\ell\in\NN)$. We proceed by a contradiction, in the same spirit as in \cite[Lemma~2.3]{FQpG}, where the author has an undirected graph. Assume that $\partial(z,y)>\partial(u,v)=m$. Then, $(A^m)_{uv}\ne 0$ and $(A^m)_{zy}=0$, a contradiction. The claim~follows. 

\bigskip
Next we show that Claim~2 holds.

\medskip
\noindent
{\sc Claim 2.} Every distance-$i$ matrix $A_i$ of $\G=\G(A)$ belongs to $\B$, i.e., $A_i\in\B$ $(0\le i\le D)$.

\smallskip
{\bf Proof of Claim 2.} From the proof of Claim~1 it follows that, if $y,z\in X$ are two arbitrary vertices such that $\partial(z,y)=i$, then there exists $B_j$ (for some $0\le j\le d$) such that $(B_j)_{zy}=1$. Recall also that $(A_i)_{zy}=1$. In fact, for such an index $j$ and any nonzero $(u,v)$-entry of $B_j$, we have $\partial(u,v)=i$. This yields
$$
A_i=\sum_{j:A_i\circ B_j\ne{\boldsymbol{O}}} B_j
\qquad(0\le i\le D).
$$
The result follows.




\section*{Acknowledgments}

This work is supported in part by the Slovenian Research Agency (research program P1-0285 and research projects J1-3001 and N1-0353).

\section*{Declaration of competing interest}

The authors declare that they have no known competing financial interests or personal relationships that could have appeared to influence the work reported in this paper.


{\small
\bibliographystyle{references}
\bibliography{stochastic}

\begin{thebibliography}{10}
\expandafter\ifx\csname urlstyle\endcsname\relax
  \providecommand{\doi}[1]{doi:\discretionary{}{}{}#1}\else
  \providecommand{\doi}{doi:\discretionary{}{}{}\begingroup
  \urlstyle{rm}\Url}\fi

\bibitem{AD}
A.~S. Asratian, T.~M.~J. Denley and R.~H\"{a}ggkvist, \emph{Bipartite graphs
  and their applications}, volume 131 of \emph{Cambridge Tracts in
  Mathematics}, Cambridge University Press, Cambridge, 1998,
  \doi{10.1017/CBO9780511984068},
  \url{https://doi.org/10.1017/CBO9780511984068}.

\bibitem{AS}
S.~Axler, \emph{Linear algebra done right}, Undergraduate Texts in Mathematics,
  Springer, Cham, 3rd edition, 2015, \doi{10.1007/978-3-319-11080-6},
  \url{https://doi.org/10.1007/978-3-319-11080-6}.

\bibitem{B}
N.~Biggs, \emph{Algebraic graph theory}, Cambridge Mathematical Library,
  Cambridge University Press, Cambridge, 2nd edition, 1993.

\bibitem{BG}
G.~Birkhoff, Three observations on linear algebra, \emph{Univ. Nac.
  Tucum\'{a}n. Revista A.} \textbf{5} (1946), 147--151.

\bibitem{BB}
B.~Bollob\'{a}s, \emph{Modern graph theory}, volume 184 of \emph{Graduate Texts
  in Mathematics}, Springer-Verlag, New York, 1998,
  \doi{10.1007/978-1-4612-0619-4},
  \url{https://doi.org/10.1007/978-1-4612-0619-4}.

\bibitem{BD}
R.~A. Brualdi and G.~Dahl, Diagonal sums of doubly stochastic matrices,
  \emph{Linear Multilinear Algebra} \textbf{70} (2022), 4946--4972,
  \doi{10.1080/03081087.2021.1901844},
  \url{https://doi.org/10.1080/03081087.2021.1901844}.

\bibitem{BGb}
R.~A. Brualdi and P.~M. Gibson, Convex polyhedra of doubly stochastic matrices.
  {I}. {A}pplications of the permanent function, \emph{J. Combinatorial Theory
  Ser. A} \textbf{22} (1977), 194--230, \doi{10.1016/0097-3165(77)90051-6},
  \url{https://doi.org/10.1016/0097-3165(77)90051-6}.

\bibitem{R}
R.~A. Brualdi and H.~J. Ryser, \emph{Combinatorial matrix theory}, volume~39 of
  \emph{Encyclopedia of Mathematics and its Applications}, Cambridge University
  Press, Cambridge, 1991, \doi{10.1017/CBO9781107325708},
  \url{https://doi.org/10.1017/CBO9781107325708}.

\bibitem{CFfg}
M.~C\'{a}mara, J.~F\`abrega, M.~A. Fiol and E.~Garriga, Some families of
  orthogonal polynomials of a discrete variable and their applications to
  graphs and codes, \emph{Electron. J. Combin.} \textbf{16} (2009), Research
  Paper 83, 30, \doi{10.37236/172}, \url{https://doi.org/10.37236/172}.

\bibitem{CTS}
T.~S. Chihara, \emph{An introduction to orthogonal polynomials}, volume Vol. 13
  of \emph{Mathematics and its Applications}, Gordon and Breach Science
  Publishers, New York-London-Paris, 1978.

\bibitem{F}
F.~Comellas, M.~A. Fiol, J.~Gimbert and M.~Mitjana, The spectra of wrapped
  butterfly digraphs, \emph{Networks} \textbf{42} (2003), 15--19,
  \doi{10.1002/net.10085}, \url{https://doi.org/10.1002/net.10085}.

\bibitem{CFGM}
F.~Comellas, M.~A. Fiol, J.~Gimbert and M.~Mitjana, Weakly distance-regular
  digraphs, \emph{J. Combin. Theory Ser. B} \textbf{90} (2004), 233--255,
  \doi{10.1016/j.jctb.2003.07.003},
  \url{https://doi.org/10.1016/j.jctb.2003.07.003}.

\bibitem{M}
F.~Comellas and M.~Mitjana, The spectra of cycle prefix digraphs, \emph{SIAM J.
  Discrete Math.} \textbf{16} (2003), 418--421,
  \doi{10.1137/S0895480100380604},
  \url{https://doi.org/10.1137/S0895480100380604}.

\bibitem{C}
D.~M. Cvetkovi\'{c}, M.~Doob and H.~Sachs, \emph{Spectra of graphs}, Johann
  Ambrosius Barth, Heidelberg, 3rd edition, 1995, theory and applications.

\bibitem{FAb}
C.~Dalf\'{o}, M.~A. Fiol and E.~Garriga, Characterizing
  {$(\ell,m)$}-walk-regular graphs, \emph{Linear Algebra Appl.} \textbf{433}
  (2010), 1821--1826, \doi{10.1016/j.laa.2010.06.042},
  \url{https://doi.org/10.1016/j.laa.2010.06.042}.

\bibitem{FAc}
C.~Dalf\'{o}, E.~R. van Dam, M.~A. Fiol, E.~Garriga and B.~L. Gorissen, On
  almost distance-regular graphs, \emph{J. Combin. Theory Ser. A} \textbf{118}
  (2011), 1094--1113, \doi{10.1016/j.jcta.2010.10.005},
  \url{https://doi.org/10.1016/j.jcta.2010.10.005}.

\bibitem{FAe}
V.~Diego, J.~F\`abrega and M.~A. Fiol, Equivalent characterizations of the
  spectra of graphs and applications to measures of distance-regularity,
  \emph{Electron. J. Linear Algebra} \textbf{36} (2020), 629--644.

\bibitem{D}
A.~Dress and D.~Stevanovi\'{c}, Hoffman-type identities, \emph{Appl. Math.
  Lett.} \textbf{16} (2003), 297--302, \doi{10.1016/S0893-9659(03)80047-2},
  \url{https://doi.org/10.1016/S0893-9659(03)80047-2}.

\bibitem{DU}
F.~Dufoss{\'{e}} and B.~U{\c{c}}ar, Notes on {B}irkhoff--von~{N}eumann
  decomposition of doubly stochastic matrices, \emph{Linear Algebra Appl.}
  \textbf{497} (2016), 108--115, \doi{10.1016/j.laa.2016.02.023},
  \url{https://doi.org/10.1016/j.laa.2016.02.023}.

\bibitem{FaD}
M.~A. Fiol, Algebraic characterizations of distance-regular graphs, volume 246,
  pp. 111--129, 2002, \doi{10.1016/S0012-365X(01)00255-2}, formal power series
  and algebraic combinatorics (Barcelona, 1999),
  \url{https://doi.org/10.1016/S0012-365X(01)00255-2}.

\bibitem{FQpG}
M.~A. Fiol, Quotient-polynomial graphs, \emph{Linear Algebra Appl.}
  \textbf{488} (2016), 363--376, \doi{10.1016/j.laa.2015.09.053},
  \url{https://doi.org/10.1016/j.laa.2015.09.053}.

\bibitem{Fsp}
M.~A. Fiol, S.~Gago and E.~Garriga, A simple proof of the spectral excess
  theorem for distance-regular graphs, \emph{Linear Algebra Appl.} \textbf{432}
  (2010), 2418--2422, \doi{10.1016/j.laa.2009.07.030},
  \url{https://doi.org/10.1016/j.laa.2009.07.030}.

\bibitem{FGla}
M.~A. Fiol and E.~Garriga, From local adjacency polynomials to locally
  pseudo-distance-regular graphs, \emph{J. Combin. Theory Ser. B} \textbf{71}
  (1997), 162--183, \doi{10.1006/jctb.1997.1778},
  \url{https://doi.org/10.1006/jctb.1997.1778}.

\bibitem{FGJ}
M.~A. Fiol, E.~Garriga and J.~L.~A. Yebra, Locally pseudo-distance-regular
  graphs, \emph{J. Combin. Theory Ser. B} \textbf{68} (1996), 179--205,
  \doi{10.1006/jctb.1996.0063}, \url{https://doi.org/10.1006/jctb.1996.0063}.

\bibitem{FMPS}
M.~A. Fiol and S.~Penji{\'{c}}, On symmetric association schemes and associated
  quotient-polynomial graphs, \emph{Algebr. Comb.} \textbf{4} (2021), 947--969,
  \doi{10.5802/alco}, \url{https://doi.org/10.5802/alco}.

\bibitem{HoP}
A.~J. Hoffman, On the polynomial of a graph, \emph{Amer. Math. Monthly}
  \textbf{70} (1963), 30--36, \doi{10.2307/2312780},
  \url{https://doi.org/10.2307/2312780}.

\bibitem{HMc}
A.~J. Hoffman and M.~H. McAndrew, The polynomial of a directed graph,
  \emph{Proc. Amer. Math. Soc.} \textbf{16} (1965), 303--309,
  \doi{10.2307/2033868}, \url{https://doi.org/10.2307/2033868}.

\bibitem{HJ}
R.~A. Horn and C.~R. Johnson, \emph{Matrix analysis}, Cambridge University
  Press, Cambridge, 2nd edition, 2013.

\bibitem{T}
Y.~Hou and F.~Tian, A note on {H}offman-type identities of graphs, \emph{Linear
  Algebra Appl.} \textbf{402} (2005), 143--149,
  \doi{10.1016/j.laa.2004.12.017},
  \url{https://doi.org/10.1016/j.laa.2004.12.017}.

\bibitem{mR}
D.~Margalit, J.~Rabinoff and B.~Williams, \emph{Interactive Linear Algebra: UBC
  edition}, University of British Columbia, 2023,
  \url{https://personal.math.ubc.ca/~tbjw/ila/index2.html}.

\bibitem{MCm}
C.~Meyer, \emph{Matrix analysis and applied linear algebra}, Society for
  Industrial and Applied Mathematics (SIAM), Philadelphia, PA, 2000,
  \doi{10.1137/1.9780898719512}, with 1 CD-ROM (Windows, Macintosh and UNIX)
  and a solutions manual (iv+171 pp.),
  \url{https://doi.org/10.1137/1.9780898719512}.

\bibitem{MPc}
G.~Monzillo and S.~Penji{\' c}, On commutative association schemes and
  associated (directed) graphs, 2023., \url{https://arxiv.org/abs/2307.11680}.

\bibitem{SPm}
S.~Penji{\'c}, \emph{Algebraic characterizations of distance-regular graphs},
  Master's thesis, University of Sarajevo, 2013.

\bibitem{SP}
S.~Penji{\' c}, \emph{On the {T}erwilliger algebra of bipartite
  distance-regular graphs}, {U}niversity of {P}rimorska, 2019, thesis (Ph.D.)
  -- {U}niversity of {P}rimorska, {K}oper,
  \url{https://osebje.famnit.upr.si/~penjic/research/}.

\bibitem{PM}
H.~Perfect and L.~Mirsky, The distribution of positive elements in
  doubly-stochastic matrices, \emph{J. London Math. Soc.} \textbf{40} (1965),
  689--698, \doi{10.1112/jlms/s1-40.1.689},
  \url{https://doi.org/10.1112/jlms/s1-40.1.689}.

\bibitem{SK}
R.~Sinkhorn and P.~Knopp, Concerning nonnegative matrices and doubly stochastic
  matrices, \emph{Pacific J. Math.} \textbf{21} (1967), 343--348,
  \url{http://projecteuclid.org/euclid.pjm/1102992505}.

\bibitem{Y}
Y.~Teranishi, The {H}offman number of a graph, \emph{Discrete Math.}
  \textbf{260} (2003), 255--265, \doi{10.1016/S0012-365X(02)00764-1},
  \url{https://doi.org/10.1016/S0012-365X(02)00764-1}.

\bibitem{FAf}
E.~R. van Dam and M.~A. Fiol, A short proof of the odd-girth theorem,
  \emph{Electron. J. Combin.} \textbf{19} (2012), Paper 12, 5,
  \doi{10.37236/2289}, \url{https://doi.org/10.37236/2289}.

\bibitem{L}
J.~H. van Lint and R.~M. Wilson, \emph{A course in combinatorics}, Cambridge
  University Press, Cambridge, 2nd edition, 2001,
  \doi{10.1017/CBO9780511987045},
  \url{https://doi.org/10.1017/CBO9780511987045}.

\bibitem{W}
D.~B. West, \emph{Introduction to graph theory}, Prentice Hall, Inc., Upper
  Saddle River, NJ, 1996.

\bibitem{WD}
Y.~Wu and A.~Deng, Hoffman polynomials of nonnegative irreducible matrices and
  strongly connected digraphs, \emph{Linear Algebra Appl.} \textbf{414} (2006),
  138--171, \doi{10.1016/j.laa.2005.09.012},
  \url{https://doi.org/10.1016/j.laa.2005.09.012}.

\bibitem{A}
Y.~Wu and A.~Deng, Hoffman polynomials of nonnegative irreducible matrices and
  strongly connected digraphs, \emph{Linear Algebra Appl.} \textbf{414} (2006),
  138--171, \doi{10.1016/j.laa.2005.09.012},
  \url{https://doi.org/10.1016/j.laa.2005.09.012}.

\bibitem{TtYl}
T.-T. Xia, Y.-Y. Tan, X.~Liang and J.~H. Koolen, On association schemes
  generated by a relation or an idempotent, \emph{Linear Algebra Appl.}
  \textbf{670} (2023), 1--18, \doi{10.1016/j.laa.2023.03.029},
  \url{https://doi.org/10.1016/j.laa.2023.03.029}.

\end{thebibliography}
}

\end{document}